\documentclass[11pt]{amsart}

\usepackage{amscd}
\usepackage{xypic}  
\usepackage{amssymb}
\usepackage{amsthm}
\usepackage{epsfig}
\usepackage{paralist, comment}
\usepackage{lmodern}
\usepackage[T1]{fontenc}
\usepackage[utf8]{inputenc}
\usepackage{longtable}
\usepackage{calc}
\usepackage{paralist}

\usepackage[hidelinks=true]{hyperref}
\usepackage{graphicx}

\usepackage[all,cmtip]{xy}
\usepackage{amsmath}
\usepackage{xcolor}

\newtheorem{thm}{Theorem}[section] 
\newtheorem{theorem}{Theorem}

\newtheorem{lemma}[thm]{Lemma}

\newtheorem{proposition}[thm]{Proposition}

\newtheorem{claim}[thm]{Claim}

\newtheorem{corollary}[thm]{Corollary}
\theoremstyle{definition}

\newtheorem{remark}[thm]{Remark}

 \newtheorem{notation}[thm]{Notation}
  \newtheorem{definition-remark}[thm]{Definition-Remark}

\def\geq{\geqslant}
\def\leq{\leqslant}

\def\ker{\operatorname{ker}}

\def\cork{\operatorname{cork}}

\def\min{\operatorname{min}}
\def\im{\operatorname{im}}

\def\max{\operatorname{max}}

\def\c1{\operatorname{c_1}}
\def\c2{\operatorname{c_2}}
\def\Cliff{\operatorname{Cliff}}
\def\gon{\operatorname{gon}}

\def\Hilb{\operatorname{Hilb}}

\def\pr{\operatorname{pr}}

\def\kod{\operatorname{kod}}

\def\CC{{\mathbb C}}

\def\ZZ{{\mathbb Z}}

\def\QQ{{\mathbb Q}}
\def\PP{{\mathbb P}}
\def\P{{\mathcal P}} 
\def\EC{\mathcal{EC}}
\let\P\EC
\def\KC{\mathcal{KC}}
\def\A{{\mathcal A}}

\def\D{{\mathcal D}}
\def\G{{\mathcal G}}
\def\L{{\mathcal L}}
\def\M{{\mathcal M}}
\def\N{{\mathcal N}}
\def\O{{\mathcal O}}
\def\I{{\mathcal J}}

\def\E{{\mathcal E}}
\def\T{{\mathcal T}}
\def\H{{\mathcal H}}
\def\F{{\mathcal F}}
\def\K{{\mathcal K}}

\def\C{{\mathcal C}} 
 
\def\K{{\mathcal K}}

\def\c{\mathfrak{c}}

\def\R{\mathcal{R}}

\def\x{\times}                   
\def\cong{\simeq}

\def\sub{\subseteq}

\def\+{\oplus}               
\def\*{\otimes}                  

\def\Pic{\operatorname{Pic}}
\def\NS{\operatorname{NS}}
\def\Num{\operatorname{Num}}

\begin{document}

\title{Moduli of curves on Enriques surfaces}

\author[C.~Ciliberto]{Ciro Ciliberto}
\address{Ciro Ciliberto, Dipartimento di Matematica, Universit\`a di Roma Tor Vergata, Via della Ricerca Scientifica, 00173 Roma, Italy}
\email{cilibert@mat.uniroma2.it}

\author[T.~Dedieu]{Thomas Dedieu}
\address{Thomas Dedieu,
Institut de Mathématiques de Toulouse~---~UMR5219,
Université de Toulouse~---~CNRS,
UPS IMT, F-31062 Toulouse Cedex 9, France} 
\email{thomas.dedieu@math.univ-toulouse.fr} 

\author[C.~Galati]{Concettina Galati}
\address{Concettina Galati, Dipartimento di Matematica e Informatica, Universit\`a della Calabria, via P. Bucci, cubo 31B, 87036 Arcavacata di Rende (CS), Italy}
\email{galati@mat.unical.it}

\author[A.~L.~Knutsen]{Andreas Leopold Knutsen}
\address{Andreas Leopold Knutsen, Department of Mathematics, University of Bergen, Postboks 7800,
5020 Bergen, Norway}
\email{andreas.knutsen@math.uib.no}



 \begin{abstract} We  compute the  number of moduli of all irreducible components of the moduli space of smooth curves on   Enriques surfaces. In most cases, the moduli maps to the moduli space of Prym curves  are  generically  injective or dominant. Exceptional behaviour  is related  to existence of Enriques--Fano threefolds and to curves with nodal Prym-canonical model.
 \end{abstract}

\maketitle


\vspace{-1cm}

\section{Introduction} \label{sec:intro}

Moduli of curves on projective surfaces have been the object of intensive study  for a long  time. 
In  more recent times the so-called \emph{Mukai map} $c_g$ from the $(19+g)$--dimensional moduli space of  smooth  $K3$ sections of genus $g$  (that is, pairs $(S,C)$, where $S$ is a smooth $K3$ surface and $C \subset S$ is a smooth genus $g$ curve)  to $\M_g$ has been given much attention in relation to the birational geometry of $\M_g$ and of the moduli space of $K3$ surfaces of genus $g$. In particular $c_g$ is dominant for $g\leqslant 11$ and $g\neq 10$, is birational onto its image for $g\geq 11$ and $g\neq 12$, and its image is a divisor in genus 10 and it has generically one-dimensional fibers in genus $12$ \cite{muk1,muk2,muk3,MM,clm}.  Notable are the relations of pathologies of $c_g$ with the existence of Fano and Mukai manifolds  \cite{clm2,cds}. Also  recall that  Mukai's program towards reconstructing a fiber of $c_g$ is now proven \cite{muk1,abs1,Fe},  and that  the image of $c_g$ has been recently characterized, via the Gauss--Wahl map, for Brill-Noether-Petri general curves \cite{abs,W2}.

In this paper we consider smooth curves on Enriques   surfaces. The  moduli of  such  curves have  not been
systematically investigated so far. Probably this is due to the fact
that the Enriques case is much more complicated and rich  compared to the $K3$ case  due to the
presence of many irreducible components of  the moduli space of  polarized such surfaces,  whence also of the moduli space of smooth curves on Enriques surfaces,  even when
fixing the genus of the polarization. Remarkably enough our  results
give  the  number of moduli of {\it all}  such components, equivalently,
the dimension of the image (or of a general fiber) of the moduli map.  It should
be  noted  that there are  some striking analogies with the $K3$ case,
 including  behaviour   induced  by the existence of Enriques--Fano
threefolds, as well as more exceptional behavior, e.g., related to
curves with nodal Prym--canonical models.

We now present our results. Let $\E$ denote the smooth  irreducible 
$10$-dimensional moduli space parametrizing smooth, complex Enriques
 surfaces 
and  $\E_{g,\phi}$ the  (in general reducible)  moduli space of pairs $(S,H)$ such that $S$ is a member of 
$\E$  and $H$ is  an ample  line bundle on $S$  satisfying  $H^2=2g-2$
and $\phi(H)=\phi$, where
\begin{equation} \label{eq:fi} 
 \phi(H):=\min\bigl\{E \cdot H \; | \; E \in \NS(S), E^2=0, E  > 0 \bigr\}.
\end{equation}
 Recall  that $\phi^2\leq 2g-2$  by \cite[Cor.\ 2.7.1]{cd}. 

Denote by $\P_{g,\phi}$ the  moduli  space  of   triples $(S,H,C)$ where  $(S,H)$ is a member of   $\E_{g,\phi}$  and $C\in |H|$ is a smooth irreducible curve.  Note that $\EC_{g,\phi}$  has  as  many irreducible components as $\E_{g,\phi}$.  There are natural  morphisms
\begin{equation}\label{eq:enriques}
\xymatrix{ 
&\P_{g,\phi} \ar[ld]_(.55){p_{g,\phi}}  
\ar[d]_{\chi_{g,\phi}}    
\ar[rd]^(.55){c_{g,\phi}} &\\
\E_{g,\phi}&   \R_g \ar[r]    &\M_g,
}
\end{equation}
where $\R_g$ is the moduli space of {\it Prym curves}, that is, of pairs $(C,\eta)$, with $C$ a smooth, irreducible, genus $g$ curve and $\eta$ a non--zero $2$-torsion element of $\Pic^0(C)$. 
The  map $\chi_{g,\phi}$  sends $(S,H,C)$ to the Prym curve 
$(C, \omega_S \otimes \O_C)$. The morphism
$c_{g,\phi}$  is the  composition of the latter with the  forgetful   covering map $\R_g \to \M_g$,
which has degree $2^{2g}-1$. By a dimension count, one could  expect 
$\chi_{g,\phi}$ and $c_{g,\phi}$ to be dominant (on some or all components) for $g \leq 6$ and generically finite (on some or all components) for $g \geq 6$.

As far as we know, the only known result so far concerning the maps $\chi_{g,\phi}$ and $c_{g,\phi}$ is the one of Verra \cite{Ve} stating that $\chi_{6,3}$ is dominant, equivalently generically finite (note that $\E_{6,3}$ is irreducible).

Our main results are the following.  We present the cases $\phi \geq 3$, $\phi=2$ and $\phi=1$ separately.  We refer to the  tables  in \S \ref{sec:msenr}  and Notation \ref{not:comp}  for the definition of the various components of   $\E_{g,\phi}$ and  $\P_{g,\phi}$  showing up in the results below.

\begin{theorem} \label{mainthm1}
Assume that $\phi \geq 3$ (whence $g \geq 6$).  The map
$\chi_{g,\phi}: \P_{g,\phi} \to \R_g$ is generically injective on any
irreducible component of  $\EC_{g,\phi}$  not appearing in the list
below,  for  which the dimension of a general fiber is indicated:\\[1mm]
\resizebox{\linewidth}{!}{%
\begin{tabular}{|l |c|c|c|c|c|c|c|c|c|c|}
\hline
comp. & 
$\EC_{7,3}$ &
$\EC_{9,3}^{(II)}$ &
$\EC_{9,4}^{+}$ &
$\EC_{9,4}^{-}$ &
$\EC_{10,3}^{(II)}$ &
$\EC_{13,3}^{(II)}$ &
$\EC_{13,4}^{(II)^+}$ &
$\EC_{13,4}^{(II)^-}$ &
$\EC_{17,4}^{(IV)^+}$ &
$\EC_{17,4}^{(IV)^-}$ \\
\hline
fib.dim.
&   $1$
&  $1$
& $3$
&  $0$
& $2$
& $1$
& $1$
& $0$
& $1$
& $0$ \\ \hline
\end{tabular}
}
\end{theorem}

 In particular, we obtain that
$ \chi_{6,3}: \P_{6,3} \longrightarrow \R_6$
is birational, improving the result of \cite{Ve}. Moreover, in analogy with the $K3$ case, for any $g \geq 8$ there is a component of $\E_{g,\phi}$ on which 
$\chi_{g,\phi}$ is generically injective, whereas on $\E_{7,3}$ (which is  irreducible),  the map $\chi_{g,\phi}$ has generically one--dimensional fibers. However, in contrast to the $K3$ case, there are more components of $\E_{g,\phi}$ for $g \geq 8$ where $\chi_{g,\phi}$ is not generically finite. This phenomenon can be explained by the existence of 
Enriques--Fano threefolds, see \S \ref{sec:EF-nuovo}. 

For $\phi=2$  we obtain:

 \begin{theorem} \label{mainthm2}
   The  map $\chi_{g,2}:\P_{g,2} \to \R_g$ is  generically finite on all
irreducible components of  $\P_{g,2}$  when $g \geq 10$. For $g \leq 9$ the
dimension of a general fiber of  $\chi_{g,2}$  on  the
various irreducible components of  $\P_{g,2}$  is as follows:\\[1mm]
\resizebox{\linewidth}{!}{%
\begin{tabular}{|l |c|c|c|c|c|c|c|c|c|c|c|c|}
\hline
comp. & 
$\P_{9,2}^{(I)}$ &
$\P_{9,2}^{(II)^+}$\! &
$\P_{9,2}^{(II)^-}$\! &
$\P_{8,2}$ &
$\P_{7,2}^{(I)}$ &
$\P_{7,2}^{(II)}$ &
$\P_{6,2}$ &
$\P_{5,2}^{(I)}$ &
$\P_{5,2}^{(II)^+}$\! &
$\P_{5,2}^{(II)^-}$\! &
$\P_{4,2}$ &
$\P_{3,2}$ 
\\ \hline
fib.dim.
&   $0$
& $2$
& $1$
&  $0$
& $1$
&  $3$
& $2$
&   $3$
& $6$
& $4$
&   $4$  
& $6$ \\ \hline
\end{tabular}}
\end{theorem}

In particular, $\chi_{g,2}$ is dominant precisely  on $\P_{3,2}$ and $\P_{4,2}$ 
and is generically finite on at least one component of  $\P_{g,2}$  precisely for $g \geq 8$.  The  positive-dimensional fibers of $\chi_{9,2}$ on  $\P_{9,2}^{(II)^+}$ and 
$\P_{9,2}^{(II)^-}$  can again be explained by the existence of 
Enriques--Fano threefolds, see  Corollary \ref{cor:EF1}.  The other positive-dimensional fibers are due to the fact that the image of  $\chi_{g,2}$   lies in quite special loci, as we now explain.   
Define:
 \begin{compactitem}
   \item $\R^0_g$ --- the locally closed  locus in  $\R_g$ of pairs $(C,\eta)$ for which  the complete linear system $|\omega_C(\eta)|$ is base point free and the map $C \to \PP^{g-2}$ it defines (the so--called \emph{Prym--canonical map}) is  not an embedding. This locus is irreducible (and unirational) of dimension $2g+1$ for $g \geq 5$  by \cite[Thm.~1]{cdgk2}. (Obviously, $\R^0_g$ is dense in $\R_g$ for $g \leq 4$.) Moreover, for the general element, the Prym--canonical map is birational onto its image,  which  has precisely two nodes,  cf. \cite[Prop.~1.2]{cdgk2}.
\item $\R^{0,\mathrm{nb}}_g$ --- the closed  locus in  $\R^0_g$ of pairs
$(C,\eta)$ for which the Prym--canonical map is not birational onto
its image. This locus is irreducible of dimension $2g-2$ for $g \geq
4$ and dominates the bielliptic locus in $\M_g$ via the forgetful map
$\R_g \to \M_g$ by \cite[Cor.~2.2]{cdgk2}.
\item $\D^{0}_5$ ---  the locally closed locus  in $\R^0_5$ of pairs 
$(C,\eta)$ with  $4$-nodal Prym-canonical model.    By 
\cite[Prop.~5.3]{cdgk2} this locus is an irreducible (unirational) divisor in 
$\R^0_5$  whose closure  in $\R_5$ coincides with closure of the locus of pairs 
$(C,\eta)$ carrying 
a theta-characteristic $\theta$
such that
$h^0(\theta)=h^0(\theta+\eta)=2$.  

\end{compactitem}
The image of $\chi_{g,2}$ (on   any component of  $\P_{g,2}$) always lies in 
$\R^0_g$, cf. Lemma \ref{lemma:posfibfi2}(ii)-(v), and consequently, by counting dimensions one a priori sees that $\chi_{g,2}$ has expected fiber dimension $\max\{0,8-g\}$. Furthermore, as a consequence of Theorem \ref{mainthm2}, the maps $\chi_{g,2}$ dominate some of the peculiar loci above in various cases. Indeed,
it follows from   Proposition \ref{prop:fi2II}(i)-(ii) and Corollary \ref{cor:image} that:
\begin{compactitem}
\item $\chi_{5,2}$  on  $\P_{5,2}^{(I)}$  (respectively,   $\chi_{6,2}$, $\chi_{7,2}$  on  $\P_{7,2}^{(I)}$ , $\chi_{8,2}$) dominates $\R^0_5$ (resp., $\R^0_6$, $\R^0_7$, $\R^0_8$). In particular, the image of $\chi_{5,2}$ on  $\P_{5,2}^{(I)}$  is a divisor in $\R_5$; this parallells the situation of $\im c_{10}$ in the $K3$ case.
 \item $\chi_{5,2}$ on  $\P_{5,2}^{(II)^+}$  dominates $\R^{0,\mathrm{nb}}_5$.
\item $\chi_{5,2}$ on  $\P_{5,2}^{(II)^-}$  dominates $\D^{0}_5$.
\end{compactitem}

For $\phi=1$ the moduli spaces $\E_{g,1}$ are irreducible for all $g$ and  the image of $\chi_{g,1}$ (and of $c_{g,1}$) lies in the 
hyperelliptic locus, cf.\   Lemma \ref{lemma:posfibfi2}(i),  
hence the expected fiber dimension is $\max\{10-g,0\}$. We  prove that this is indeed the dimension of a general fiber:

\begin{theorem} \label{mainthm3}
   The dimension of a general fiber of $\chi_{g,1}$ and of $c_{g,1}$ is $\max\{10-g,0\}$. Hence, $c_{g,1}$ dominates the hyperelliptic locus if $g \leq 10$ and is generically finite if $g \geq 10$.  
\end{theorem}
  
 An immediate consequence of the above results is: 

\begin{corollary}
  A general curve of genus $2$, $3$, $4$ and $6$ lies on an Enriques surface, whereas a general curve of genus $5$ or $\geq 7$ does not.
A general hyperelliptic curve of genus $g$ lies on an Enriques surface if and only if $g \leq 10$. 
\end{corollary}

 The proof of Theorem \ref{mainthm1}  also has an application to the classification of 
projective varieties having 
Enriques surfaces as linear sections. We recall that a projective variety $V \subset \PP^N$ is said to be {\it $k$-extendable} if there exists a projective variety $W \subset \PP^{N+k}$, different from a cone, such that $V =W \cap \PP^N$ (transversely). The question of $k$-extendability of Enriques surfaces is still open, although it is proved in \cite{Pr,KLM} that $N \leq 17$ is a necessary condition for $1$--extendability, and {\it terminal threefolds} having Enriques surfaces as hyperplane sections have been classified in  \cite{bay,sa,Mi}.

\begin{corollary} \label{cor:EF-emb}
  Let $S \subset \PP^N $ be an Enriques surface not containing any smooth rational curve. If $S$ is $1$--extendable, then $(S,\O_S(1))$ belongs to the following list:
\[ \E_{17,4}^{(IV)^+}, \; \E_{13,4}^{(II)^+}, \; \E_{13,3}^{(II)}, \; \E_{10,3}^{(II)}, \; \E_{9,4}^{+}, \; \E_{9,3}^{(II)}, \;  \E_{7,3}.\] 
Furthermore, the members of this list are all at most $1$--extendable, except for members of $\E_{10,3}^{(II)}$, which are at most $2$--extendable, and of $\E_{9,4}^{+}$, which are at most $3$--extendable.
\end{corollary}

 This result is sharp in the case of $1$-extendability:  The general members of the moduli spaces of Corollary \ref{cor:EF-emb} indeed occur as hyperplane sections of threefolds different from cones, cf. Remark  \ref{rem:altricasi}.  Furthermore, one cannot remove the assumption about $S$ not containing smooth rational curves, as  there are  threefolds  different from cones  enjoying the peculiar property that  their Enriques hyperplane sections  belong  to $\E_{8,3}$ and $\E_{6,3}$  and   contain a smooth rational curve, cf. Corollary \ref{cor:nodalbayle}. We remark that the proof of Corollary \ref{cor:EF-emb} is independent from, and  much simpler than, the results in  \cite{Pr,KLM}, but it needs the technical assumption  about rational  curves, which can probably be avoided, at the expense of adding more  cases, cf.  Remark \ref{rem:nodnotnec}. 
 We  refer to Corollary \ref{cor:EF} for another variant of Corollary \ref{cor:EF-emb}.

 Our general strategy is to compute the kernel of the differential of
the map $c_{g,\phi}$, see \S \ref{sec:gen}; to this end
we develop in \S \ref{sec:tools} tools to compute the cohomology of
twisted tangent bundles on Enriques surfaces.
In some cases additional arguments are required, involving for
instance extensions to Enriques--Fano threefolds (see \S
\ref{sec:EF-nuovo}), and specializations to  Enriques surfaces containing smooth rational curves (see
\S\S \ref{sec:fi2-II-n}--\ref{sec:fi2-II}).
 Theorem \ref{mainthm1} and Corollary \ref{cor:EF-emb}  are  proved
in \S \ref{sec:enr};  Theorem \ref{mainthm2} is obtained by
combining Propositions \ref{prop:fi2gpari}, \ref{prop:fi2g>5} and
\ref{prop:fi2II};  Theorem \ref{mainthm3}  is  proved in \S
\ref{sec:fi2-II}.

In conclusion we remark that our work leaves several interesting open questions. For example: is it possible to characterize curves on Enriques surfaces in terms of the suitable Gauss--Prym map? In the cases of generic injectivity of $\chi_{g,\phi}$, is it possible to develop an analogue of Mukai's programme of explicit reconstruction of the Enriques surface from its Prym curve section?   The latter question was proposed to us by Enrico Arbarello. Finally, in view of 
Corollary \ref{cor:EF-emb}, are the general members of 
$\E_{10,3}^{(II)}$ (respectively, $\E_{9,4}^{+}$) $2$--extendable (resp., $3$--extendable)? 

\vspace{0.3cm} {\it Acknowledgements.} The authors thank Alessandro
Verra  and Enrico Arbarello  for useful conversations on the subject and acknowledge funding
from MIUR Excellence Department Project CUP E83C180 00100006 (CC),
project FOSICAV within the  EU  Horizon
2020 research and innovation programme under the Marie
Sk{\l}odowska-Curie grant agreement n.  652782 (CC, ThD),
 GNSAGA of INDAM (CC,CG),  the Trond Mohn  
Foundation (ThD, ALK) and grant 261756 of the Research Council of Norway (ALK).
 Finally the authors wish to thank the referee for the extremely careful reading of the paper and for her/his useful comments.

\section{Moduli spaces of Enriques surfaces}
\label{sec:msenr}

We  first  briefly recall some well-known properties of divisors on Enriques surfaces. 

Any irreducible curve $C$ on an Enriques surface satisfies $C^2 \geq -2$, with equality occurring if and only if  $C \cong \PP^1$.  
 The latter   curves are  called {\it nodal}, and Enriques surfaces containing (respectively, not containing)  them  are  called {\it nodal} (resp., {\it unnodal}). 
It is well-known that the general Enriques surface is unnodal, cf.\  references in \cite[p.~577]{cos2}. 

 A divisor  $E$ is said to be {\it isotropic} if $E^2=0$ and $E \not \equiv 0$   (where '$\equiv$' denotes numerical equivalence)  and {\it primitive}  if it is non-divisible in $\Num S$. 
 If $E$ is primitive, isotropic and nef, then  $|2E|$ is a base point free pencil with general member a smooth elliptic curve, cf.\   \cite[Prop. 3.1.2]{cd}. In  this   case,  $\dim( |E|)=0$ and $E$ is called a {\it half-fiber}, cf.\  \cite[p.~172]{cd}.
Conversely, any elliptic pencil $|P|$ contains precisely two double fibers $2E$ and $2E'$,  where $E'$ is the only member of $|E+K_S|$.  It is clear that, when $H$ is  big and  nef, the invariant $\phi(H)$ in \eqref{eq:fi}  is computed by a half-fiber.

Let $\E$ (resp. $\K^{\iota}$) denote the $10$-dimensional smooth
moduli space parametrizing smooth Enriques surfaces (respectively,
smooth $K3$ surfaces with a fixed point free involution), and let $\K$
denote the $20$-dimensional moduli space parametrizing smooth $K3$
surfaces,  cf.\ \cite[{VIII.\S12,\S\S 19-21}]{BPHV}.  We have a natural  bijective map sending a $K3$ surface with fixed point free involution to the quotient surface by the involution  
\begin{equation}
  \label{eq:dcmod}
  \delta: \K^{\iota} \longrightarrow \E. 
\end{equation}

Let $\E_{g,\phi}$ (respectively, $\widehat{{\E}}_{g,\phi}$) denote the
moduli space of  {\it polarized} (resp., {\it numerically polarized})
{\it Enriques surfaces}, that is, pairs $(S,H)$ (resp., $(S,[H])$)
such that $[S] \in \E$ and $H \in \Pic(S)$ (resp., $[H] \in \Num(S)$)
is  ample  with $H^2=2g-2 \geq 2$ and $\phi=\phi(H)$.
There is an  \'etale  cover  $\E_{g,\phi} \to \E$,
and $\E_{g,\phi}$ is smooth. There is also an  \'etale double  cover 
$\rho:  \E_{g,\phi} \to \widehat{\E}_{g,\phi}$ mapping $(S,H)$ and ($S,H+K_S)$ to $(S,[H])$.  We refer to \cite[\S 2]{cdgk} for details and references and also recall  that $\phi(H)^2 \leq H^2$ by \cite[Cor.\ 2.7.1]{cd}.

The spaces $\E_{g,\phi}$ need not be irreducible. In \cite{cdgk}
various irreducible components were determined and their
unirationality or uniruledness was proved. In particular, all
components are determined and described for $\phi \leq 4$ and $g \leq
20$ respectively. The description is in terms of isotropic
decompositions, as we now explain.

By  \cite[Cor. 4.6, Cor. 4.7, Rem. 4.11]{cdgk} 
any effective line bundle $H$ with $H^2\geq 0$ on an Enriques surface  $S$  can be written as (denoting linear equivalence by '$\sim$'):  
\begin {equation}\label{eq:ssid}  H  \sim   a_1E_1+\cdots+a_nE_{n}  + \varepsilon K_S 
\end{equation}
where all $E_i$ are  effective, primitive  and isotropic, all $a_i$ are positive integers, $n \leq 10$, 
\[
\varepsilon= \begin{cases} 0, & \mbox{if $H+K_S$ is not $2$-divisible in $\Pic(S)$,} \\
1, & \mbox{if $H+K_S$ is $2$-divisible in $\Pic(S)$,} 
\end{cases}
\]
 and moreover 
\begin{equation}\label{eq:varie}
\begin{cases}
\mbox{either  $n \neq 9$,  $E_i \cdot E_j=1$ for all $i \neq j$,} \\ 
\mbox{or  $n \neq 10$,  $E_1 \cdot E_2=2$ and otherwise $E_i \cdot E_j=1$ for all $i
\neq j$,} \\
\mbox{or $E_1 \cdot E_2=E_1 \cdot E_3=2$ and otherwise $E_i \cdot
E_j=1$ for all $i \neq j$.}
\end{cases}
\end{equation}
We call this a {\it simple isotropic decomposition} (up to reordering indices), cf.\  \cite{cdgk}.

We say that two polarized Enriques surfaces $(S,H)$ and $(S',H')$ in $\E_{g,\phi}$ {\em admit the same simple decomposition type} if one has two simple isotropic decompositions
\[H \sim a_1 E_1+\cdots +a_nE_n+\varepsilon K_S
\; \; \mbox{and} \; \; H' \sim a_1 E'_1+\cdots +a_nE'_n+ \varepsilon  K_{S'}
\]
   and   $E_i \cdot E_j=E'_i \cdot E'_j$ for all  $i \neq j$.
 This defines an equivalence relation on $\E_{g,\phi}$ by \cite[Prop. 4.15]{cdgk}.

By \cite[Cor. 1.3 and 1.4]{cdgk}  the irreducible
components of $\E_{g,\phi}$ when $\phi \leq 4$ or $g \leq 20$
correspond precisely to the loci consisting of pairs $(S,H)$ admitting
the same decomposition type.  Moreover, by \cite[Cor. 1.5]{cdgk},  in the same range,  for $\C \subset \E_{g,\phi}$ any irreducible component, $\rho^{-1}(\rho(\C))$ is reducible if and only if $\C$ parametrizes pairs $(S,H)$ such that $H$ is $2$-divisible in $\Num(S)$.  The various irreducible
components of $\widehat{\E}_{g,\phi}$ were labeled by roman numbers in
the appendix of \cite{cdgk}. We will use the same labels for the
irreducible components of $\E_{g,\phi}$, adding a superscript ``$+$''
and ``$-$'' in the cases there are two irreducible components lying
above one irreducible component of $\widehat{\E}_{g,\phi}$. We also
adopt the following from \cite{cdgk}:

\begin{notation} \label{not:int}
  When writing a simple isotropic decomposition \eqref {eq:ssid} verifying \eqref {eq:varie} (up to permuting indices), we will adopt the convention that $E_i$, $E_j$, $E_{i,j}$ are primitive isotropic satisfying
$E_i \cdot E_j=1$ for $i \neq j$, $E_{i,j} \cdot E_i= E_{i,j} \cdot E_j=2$ and
$E_{i,j} \cdot E_k=1$ for $k \neq i,j$. 
\end{notation}

In particular, we  recall  the following (cf.\  \cite[Cor.~1.3 and Lemma  4.18]{cdgk}):
\begin{compactitem}
\item $\E_{g,1}$ is irreducible and unirational, and $H \sim (g-1)E_1+E_2$.
\item If $g$ is even (resp., $g=3$), then $\E_{g,2}$ is irreducible and unirational, and
$H \sim \frac{g-2}{2}E_1 +E_2 +E_3$ (resp., $H \sim E_1+E_{1,2}$).
\item If 
$g \geq 7$ and $g \equiv 3 \; \mbox{mod} \; 4$, then $\E_{g,2}$ has two
irreducible, unirational components
$\E_{g,2}^{(I)}$ and $\E_{g,2}^{(II)}$ corresponding, respectively,  to simple decomposition types:
\begin{itemize}
\item[(I)] $H \sim \frac{g-1}{2}E_1 +E_{1,2}$,
\item[(II)] $H \sim \frac{g-1}{2}E_1 +2E_2$.
\end{itemize}
\item If $g \geq 5$ and $g \equiv 1 \; \mbox{mod} \; 4$, then $\E_{g,2}$ has three
irreducible, unirational components
$\E_{g,2}^{(I)}$, $\E_{g,2}^{(II)^+}$ and $\E_{g,2}^{(II)^-}$,
corresponding, respectively, to simple decomposition types 
\begin{itemize}
\item[(I)\; ] \hspace{0.033cm} $H \sim \frac{g-1}{2}E_1 +E_{1,2}$,
\item[(II)$^+$] $H \sim \frac{g-1}{2}E_1 +2E_2$,
\item[(II)$^-$] $H \sim \frac{g-1}{2}E_1 +2E_2+K_S$,
\end{itemize}

\end{compactitem}

 For later reference we  list all irreducible components  of $\E_{g,\phi}$  for $\phi \geq 2$  and $g \leq 10$,
cf.\  \cite[Appendix]{cdgk}:
\begin{center}\tiny
\begin{tabular}[t]{ |l | l |l|l|l|l|}
  \hline			
  $g$ &  $\phi$  & comp.  & dec. type   
\\\hline\hline
$3$ & $2$ & $\E_{3,2}$ &  $H \sim E_1+E_{1,2}$ \\\hline\hline
$4$ & $2$ & $\E_{4,2}$ & $H \sim E_1+E_2+E_3$ \\\hline\hline
$5$ & $2$ & $\E_{5,2}^{(I)}$ &  $H \sim 2E_1+E_{1,2}$ \\
$5$ & $2$ & $\E_{5,2}^{(II)^+}$ &  $H \sim 2E_1+2E_2$  \\
$5$ & $2$ & $\E_{5,2}^{(II)^-}$ &  $H \sim 2E_1+2E_2+K_S$ \\ \hline\hline
$6$ & $2$ & $\E_{6,2}$ & $H \sim 2E_1+E_2+E_3$ \\ \hline
$6$ & $3$ & $\E_{6,3}$ & $H \sim E_1+E_2+E_{1,2}$ \\ \hline\hline 
$7$ & $2$ & $\E_{7,2}^{(I)}$ &  $H \sim 3E_1+E_{1,2}$ \\
$7$ & $2$ & $\E_{7,2}^{(II)}$ &  $H \sim 3E_1+2E_2$  \\\hline
$7$ & $3$ & $\E_{7,3}$ &  $H \sim E_1+E_2+E_3+E_4$ \\\hline\hline
$8$ & $2$ & $\E_{8,2}$ &  $H \sim 3E_1+E_2+E_3$ \\\hline
$8$ & $3$ & $\E_{8,3}$ &  $H \sim 2E_1+E_3+E_{1,2}$ \\\hline
\end{tabular}\hspace{.5cm}
\begin{tabular}[t]{ |l | l |l|l|l|l|}
   \hline			
$g$ &  $\phi$  & comp.  & dec. type   
\\\hline\hline
$9$ & $2$ & $\E_{9,2}^{(I)}$ &  $H \sim 4E_1+E_{1,2}$ \\
$9$ & $2$ & $\E_{9,2}^{(II)^+}$ &  $H \sim 4E_1+2E_2$  \\
$9$ & $2$ & $\E_{9,2}^{(II)^-}$ &  $H \sim 4E_1+2E_2+K_S$ \\ \hline
$9$ & $3$ & $\E_{9,3}^{(I)}$ &  $H \sim 2E_1+E_2+E_{1,2}$ \\
$9$ & $3$ & $\E_{9,3}^{(II)}$ &  $H \sim 2E_1+2E_2+E_3$\\ \hline 
$9$ & $4$ & $\E_{9,4}^{+}$ &  $H \sim 2(E_1+E_{1,2})$ \\  
$9$ & $4$ & $\E_{9,4}^{-}$ &  $H \sim 2(E_1+E_{1,2})+K_S$ \\ \hline\hline
$10$ & $2$ & $\E_{10,2}$ &  $H \sim 4E_1+E_2+E_3$ \\ \hline 
$10$ & $3$ & $\E_{10,3}^{(I)}$ &  $H \sim 2E_1+E_2+E_3+E_4$ \\
$10$ & $3$ & $\E_{10,3}^{(II)}$ &   $H \sim 3(E_1+E_2)$ \\  \hline 
$10$ & $4$ & $\E_{10,4}$ &  $H \sim 2E_{1,2} + E_1+E_2$ \\ \hline
\end{tabular}
\end{center}
\smallskip

\noindent We also  list all  irreducible components of 
$\E_{13,3}$, $\E_{13,4}$ and $\E_{17,4}$:
\begin{center}\tiny
\begin{tabular}[t]{ |l | l |l|l|l|l|}
  \hline			
  $g$ &  $\phi$  & comp.  & dec. type   
\\\hline\hline
$13$ & $3$ & $\E_{13,3}^{(I)}$ & $H \sim 3E_1+E_2+E_3+E_4$\\
$13$ & $3$ & $\E_{13,3}^{(II)}$ & $H \sim 4E_1+3E_2$ \\ \hline 
$13$ & $4$ & $\E_{13,4}^{(I)}$ & $H \sim 2E_1+2E_2+E_{1,2}$\\
$13$ & $4$ & $\E_{13,4}^{(II)^+}$ & $H \sim 2(E_1+E_2+E_3)$\\
$13$ & $4$ & $\E_{13,4}^{(II)^-}$ & $H \sim 2(E_1+E_2+E_3)+K_S$\\
$13$ & $4$ & $\E_{13,4}^{(III)}$ &  $H \sim 3E_1+2E_{1,2}$\\\hline
\end{tabular}
\hspace{.5cm}
\begin{tabular}[t]{ |l | l |l|l|l|l|}
  \hline			
$g$ &  $\phi$  & comp.  & dec. type   
\\\hline\hline
$17$ & $4$ & $\E_{17,4}^{(I)}$ & $H \sim 3E_1+2E_2+2E_3$\\
$17$ & $4$ & $\E_{17,4}^{(II)}$ & $H \sim 3E_1+2E_2+E_{1,2}$ \\
$17$ & $4$ & $\E_{17,4}^{(III)^+}$ & $H \sim 4E_1+2E_{1,2}$\\
$17$ & $4$ & $\E_{17,4}^{(III)^-}$ & $H \sim 4E_1+2E_{1,2}+K_S$\\
$17$ & $4$ & $\E_{17,4}^{(IV)^+}$ &  $H \sim 4E_1+4E_2$\\
$17$ & $4$ & $\E_{17,4}^{(IV)^-}$ &  $H \sim 4E_1+4E_2+K_S$\\\hline
\end{tabular}
\end{center}

\section{Generalities on moduli maps} \label{sec:gen}

Recall that if a divisor $H$ on an Enriques surface $S$  is big
and nef such that $H^2=2g-2$, then   $\dim |H|=g-1$ and  a general member $C$ of $|H|$ is a smooth irreducible curve of genus $g$   if
either $g >2$, or $g=2$ and $S$ is unnodal  or $H$ is ample, by
 \cite[Prop.~2.4]{cos2} and \cite[Thm.~4.1 and Prop.~ 8.2]{cos1}.  
 As we explained in the  introduction, one could expect
$\chi_{g,\phi}$ and $c_{g,\phi}$ from diagram \eqref {eq:enriques} to be dominant (on some or all
irreducible components) for $g \leq 6$ and generically finite (on some
or all irreducible components) for $g \geq 6$.  This expectation fails
in the cases $\phi=1,2$ for low genera as the curves in $|H|$ are all
special from a Brill-Noether theoretical point of view, cf.\ Lemma
\ref{lemma:posfibfi2} below.  It also fails in case of existence of
Enriques--Fano threefolds, as we will see in \S
\ref{sec:EF-nuovo} below.

Recalling the map \eqref {eq:dcmod}, set
$\K_{g,\phi}^{\iota}=\delta^{-1}(\E_{g,\phi})$; thus
$\K_{g,\phi}^{\iota}$ is a component of the moduli space of polarized
$K3$ surfaces $(\widetilde{S},\widetilde{H},\iota)$ of genus $2g-1$
with a fixed point free involution $\iota$, and  we have a generically injective morphism $\alpha_{g,\phi}:\K_{g,\phi}^{\iota} \to \K_{2g-1}$ forgetting the involution, where $\K_{2g-1}$    denotes the  moduli
space of polarized $K3$ surfaces of genus $2g-1$  (as the general $K3$ surface with a fixed point free involution contains only one such).  
We have the  commutative  diagram 
\begin{equation}\label{eq:involution}
\xymatrix@C=10pt@R=15pt{ 
& \M_{2g-1} & \\ 
\KC_{g,\phi}^{\iota} \ar[ru]^{c_{g,\phi}^{\iota}} \ar[d]_{p_{g,\phi}^{\iota}}
\ar[rr]^{\widetilde{\alpha}_{g,\phi}} & & \KC_{2g-1} \ar[lu]_{c_{2g-1}} \ar[d]^{q_{2g-1}}\\
\K_{g,\phi}^{\iota} \ar[rr]_{\alpha_{g,\phi}} & & \K_{2g-1} }
\end{equation}
 where $\KC_{g,\phi}^{\iota}$ is the
 moduli space of quadruples $(\widetilde{S},\widetilde{H},\iota, \widetilde{C})$, with $(\widetilde{S},\widetilde{H},\iota)$ in $\K_{g,\phi}^{\iota}$ and $\widetilde{C} \in |\widetilde{H}|$ is a smooth curve invariant under the involution $\iota$, the map $\widetilde{\alpha}_{g,\phi}$ forgets $\iota$, 
and $\KC_{2g-1}$ is the
moduli space of triples $(X,L,Y)$, with $(X,L)$ in $\K_{2g-1}$ and $
Y \in |L|$ is a smooth curve. 

Recall  now, for any smooth $C \in |H|$,  the sheaf $\T_S\langle C \rangle$ defined by
 \begin{equation}
  \label{eq:deftlc} 
\xymatrix{  0 \ar[r] & \T_S\langle C \rangle \ar[r] & \T_S \ar[r] & \N_{C/S} \ar[r] & 0,}
\end{equation}
and fitting into the exact sequence
\begin{equation}
  \label{eq:deftlc2}
\xymatrix{  0 \ar[r] &  \T_S(-C) \ar[r] & \T_S\langle C \rangle \ar[r] & \T_C \ar[r] & 0.}
\end{equation}
We have the following, cf.\  \cite[\S 3.4.4]{Ser} or \cite{beau}:

\begin{lemma} \label{lemma:diff}
  The differential of $c_{g,\phi}$ at $(S,H, C)$ (resp., of $c_{2g-1}$ at $(\widetilde{S}, \widetilde{H},\widetilde{C})$) is the morphism $H^1(\T_S\langle C \rangle) \to H^1(\T_C)$ (resp., $H^1(\T_{\widetilde{S}}\langle \widetilde{C} \rangle) \to H^1(\T_{\widetilde{C}})$) induced by \eqref{eq:deftlc2}. Its kernel is $H^1(\T_S(-C))$ (resp., 
$H^1(\T_{\widetilde{S}}
(-\widetilde{C}))$).
\end{lemma}

The spaces $H^1(\T_S(-C))$ and $H^1(\T_{\widetilde{S}}
(-\widetilde{C}))$ in the lemma are related in the following way.
Let $\pi:\widetilde{S} \to S$ be the $K3$ double cover and set $\widetilde{H}:=\pi^*H$. 
As $\pi$ is \'etale, we have $\pi^*\T_S \cong \T_{\widetilde{S}}$.
 Therefore,
\begin{equation}
  \label{eq:h1soprasotto}
  H^1 (\T_{\widetilde{S}}(-\widetilde{H}))=H^1(\T_S(-H)) \+ H^1(\T_S(-H+K_S)).
\end{equation}

\begin{lemma} \label{lemma:geninj}
 Assume that $\phi \geq 3$ (whence $g \geq 6$). Let $(C,K_S \* \O_{C})$ be  a general  element of the image of $\chi_{g,\phi}$. Denote by $\widetilde{C} \to C$ its induced double cover. If $c_{2g-1}^{-1}(\widetilde{C})$ is finite, then it consists of only one point, and also 
$\chi_{g,\phi}^{-1}((C,K_S \* \O_{C}))$ consists of only one point.
\end{lemma}

\begin{proof}
Thanks to the  bijective  map $\delta$ in \eqref{eq:dcmod}  and $\alpha_{g,\phi}$ being generically injective, the fact that $\chi_{g,\phi}^{-1}((C,K_S \* \O_{C}))$  is a point is equivalent to the fact that 
 $({c}^{\iota}_{g,\phi})^{-1}(\widetilde{C})$  is a point, where
 ${c}^{\iota}_{g,\phi}$   is as in \eqref{eq:involution}. The latter
will follow  if  $c_{2g-1}^{-1}(\widetilde{C})$  is a point. By
\cite{clm}, this property  follows if  $\widetilde{C}$  has a corank
one Gauss-Wahl map, cf.\  \cite[Sketch of proof of
Prop. 3.3]{LC}. Since $2g-2=C^2 \geq \phi(C)^2 \geq 9$ (using
\cite[Cor. 2.7.1]{cd}), we have $g \geq 6$, hence $2g-1 \geq
11$. Therefore,  if $\Cliff(\widetilde{C})\geq 3$, the fiber
$c_{2g-1}^{-1}(\widetilde{C})$ is positive dimensional as soon as
the Gauss-Wahl map of $\widetilde{C}$ has corank $>1$,  cf.\  \cite[Thm. 2.6]{cds}.  Hence,  $c_{2g-1}^{-1}(\widetilde{C})$  consists of exactly one point  if  it is finite. 

We   thus have  left to prove that $\Cliff(\widetilde{C})\geq 3$.  As  the Clifford index is constant  among  smooth  curves in the linear system $|\tilde H|$ (see \cite{gl}), we may assume that $C$ is general in its linear system. Furthermore, $\Cliff(\widetilde{C})\geq 3$ is equivalent to $\gon (\widetilde{C}) \geq 5$,  which is satisfied if $\gon (C) \geq 5$.  The cases with $\gon (C) < 2\phi(H)$ are classified in \cite[Cor. 1.5]{KL1} and a direct  check  shows  that $\gon (C) \geq 5$ when $\phi \geq 3$ and $g \geq 7$.  If $g=6$, we use the assumption that $S$ is general, 
 so that $\gon(\widetilde{C})=2\phi(C)=6$ by 
\cite[Thm. 1.1]{Ra}. 
\end{proof}

\begin{corollary} \label{cor:geninj}
Let $(S,H)$ be a general element of an irreducible component $\E'_{g,\phi}$ of $\E_{g,\phi}$ and let $\chi'_{g,\phi}$ denote the restriction of $\chi_{g,\phi}$ to $p_{g,\phi}^{-1}(\E'_{g,\phi})$. 
\begin{itemize}
\item[(i)] If $\phi \geq 3$ and $h^1(\T_S(-H))=h^1(\T_S(-H+K_S))=0$, then 
$\chi'_{g,\phi}$ is generically injective.
\item[(ii)] If $h^1(\T_S(-H))=0$, then 
$\chi'_{g,\phi}$ is generically finite.
\end{itemize}
In any case, the dimension of a general fiber of $\chi'_{g,\phi}$ is $h^1(\T_S(-H))$.
\end{corollary}

\begin{proof}
This follows from Lemmas \ref{lemma:geninj} and \ref{lemma:diff}, as
well as \eqref{eq:h1soprasotto}.
\end{proof}

In the rest of the paper we will  adopt the following:

\begin{notation} \label{not:comp}
  For any irreducible component $\E'_{g,\phi}$ of $\E_{g,\phi}$  we  express  the irreducible component
$p_{g,\phi}^{-1}(\E'_{g,\phi})$, as well as the restrictions of the maps $\chi_{g,\phi}$ and $c_{g,\phi}$ to this irreducible component, by the same  superscripts  as the ones used to label $\E'_{g,\phi}$. For instance, we set $\P_{5,2}^{(II)}:=p_{5,2}^{-1}(\E_{5,2}^{(II)})$, $c_{5,2}^{(II)}:=c_{5,2}|_{\P_{5,2}^{(II)}}$ and $\chi_{5,2}^{(II)}:=\chi_{5,2}|_{\P_{5,2}^{(II)}}$.
\end{notation}

We finish this section with a lemma  that  will be needed later.   We refer to the introduction for the definitions of the loci $\R^0_g$, $\R^{0,\mathrm{nb}}_5$ and $\D^0_5$.

\begin{lemma} \label{lemma:posfibfi2}
 
 (i) For any $g \geq 2$ the image of $c_{g,1}$ lies in the hyperelliptic locus; 
in particular the  fiber dimension is $\geq  \max\{0,10-g\}$.

(ii) The image of $\chi_{5,2}^{(I)}$ lies in  $\R^0_5$;  in particular the  fiber dimension is  $ \geq  3$.  

(iii) The image of  $\chi_{5,2}^{(II)^+}$  lies in  $\R^{0,\mathrm{nb}}_5$;  in particular the 
 fiber dimension is $ \geq   6$. 

(iv) The image of  $\chi_{5,2}^{(II)-}$ lies in $\D^0_5$;  in particular the  fiber dimension is $ \geq  4$. 

 (v) For any $g \geq 6$ the image of $\chi_{g,2}$ restricted to any component of $\P_{g,2}$  lies in  $\R^0_g$;  in particular the  fiber dimension is $ \geq  \max\{0,8-g\}$.
\end{lemma}

\begin{proof}
  Item (i) follows from \cite[Prop. 4.5.1, Cor. 4.5.1]{cd}, items (ii) and (v) from  \cite[Ex.~5.1]{cdgk2}, item (iii)   from  \cite[Rem.~5.5]{cdgk2}  and (iv) from  \cite[Ex.~5.2]{cdgk2}. 
\end{proof}

\section{Fibers of the moduli maps and    Enriques--Fano threefolds} \label{sec:EF-nuovo}

An {\it Enriques--Fano threefold  of genus $g$}  is a pair $(X,\L)$ where $X$ is a normal threefold  and $\L$ is an  ample line bundle  on $X$  with $\mathcal L^3=2g-2$  such that  $|\mathcal L|$ contains a smooth Enriques surface $S$, and  $X$ is not a generalized cone over $S$, that is,  $X$ is not isomorphic to a variety obtained by contracting to a point a negative section of some $\PP^1$-bundle over $S$.     
Such threefolds with terminal singularities are classified in \cite{bay,sa,Mi}, and examples with canonical, nonterminal singularities are  given  in 
\cite{KLM,Pr}, but a full classification of these threefolds is still missing,  although it is proved in \cite{KLM,Pr} that $g \leq 17$.  We say that a polarized Enriques surface $(S,H)$ is {\it extendable} to an Enriques--Fano threefold $(X,\L)$ if $S\in |\mathcal L|$ with  $H=\L|_S$.

\begin{lemma} \label{lemma:basicEF}
  Let $(X,\L)$ be an Enriques--Fano threefold  of genus $g$, $\pi:\widetilde{X} \to X$ a  desingularization and $S \in |\L|$ a smooth surface. Then
  \begin{itemize}
  \item[(i)] $h^0(X,\L)=g+1$ and the restriction map $H^0(X,\L) \to H^0(S,\L|_S)$ is onto;
\item[(ii)] $h^1(\O_{\widetilde{X}})=0$ and $H^0(\widetilde{X},\pi^*\L) \cong H^0(X,\L)$. 
  \end{itemize}
\end{lemma}

\begin{proof}
  Since $\pi$ is an isomorphism outside the singular locus of $X$, 
we may identify $S$ with $\pi^{-1}(S)$. By the fact that $\pi^*\L$ is big and nef
and 
\begin{equation} \label{eq:desing}
\xymatrix{ 0 \ar[r] & \O_{\widetilde{X}}(-\pi^*\L) \ar[r] &  \O_{\widetilde{X}} \ar[r] &  \O_S \ar[r] &  0,}
\end{equation} 
 we get
$h^1(\O_{\widetilde{X}})=0$. The rest of (ii) follows from the normality of $X$. Tensoring \eqref{eq:desing} by $\pi^*\L$ and taking cohomology, we get (i).
\end{proof}

 In particular, part (i) implies that $|\L|$ is base point free
if and only if $|\L|_S|$ is, for any smooth Enriques surface $S \in |\L|$, which holds if and only if
$\phi(\L|_S) \geq 2$ by 
\cite[Thms. 4.1]{cos1} or \cite[Thm. 4.4.1]{cd}. Similarly, the morphism $\varphi_{\L}$ defined by $|\L|$ is an isomorphism on $S$ if and only if $\phi(\L|_S) \geq 3$ by \cite[Thms. 5.1]{cos1} or \cite[Thm. 4.6.1]{cd} (since $\L|_S$ is ample), in which case we get that $\varphi_{\L}(X) \subset \PP^g$ is a (possibly non-normal) threefold whose general hyperplane section is a smooth Enriques surface.

 The connection to the topic of this paper is given by:  

\begin{proposition}
\label{prop:pencil}
Let $(X,\L)$ be an
Enriques--Fano threefold of genus $g  \geq 6$. Let $S\in |\mathcal L|$ be
general, and $C\in |\mathcal L_{|S}|$ be  general, with 
$\phi=\phi(\mathcal L_{|S})  \geq 2$. Then the dimension of the fiber of
$c_{g,\phi}$ at $(S,\mathcal L_{|S}, C)$ is at least 1.
\end{proposition}

\begin{proof} Consider the linear pencil $\mathfrak l$ in  $|\mathcal L|$ with base locus $C$, so that $S\in \mathfrak l$. Consider the open subset $U$ of $\mathfrak l$ whose points correspond to smooth sections of $X$. We claim that two general points of $U$ correspond to non--isomorphic  polarized  Enriques surfaces $(S',\mathcal L_{|S'}), (S'', \mathcal L_{|S''})$. The assertion clearly follows from this claim.

To prove the claim, suppose, to the contrary, that all points of $U$ correspond to isomorphic  polarized  Enriques surfaces.   This  implies that two general members in $|\L|$ are isomorphic as  polarized  Enriques surfaces.  Since $g \geq 6$ and $\phi \geq 2$, Lemma \ref{lemma:basicEF} together with
\cite[Thms. 4.1 and 5.1]{cos1} or \cite[Thm. 4.4.1 and Prop. 4.7.1]{cd} yield that 
the  map $\varphi_\L$ determined by $|\mathcal L|$ is a morphism 
that
maps $X$  birationally onto its image, which is not a cone. Hence, 
two general hyperplane sections of $Y=\varphi_\L(X)$ are projectively equivalent.  By  \cite[Prop. 1.7]{Pa}  (which applies in fact to all varieties different from cones) 
this would imply that the general hyperplane section of $Y$ is ruled, a contradiction.  \end{proof}

\begin{corollary} \label{cor:EF1}
The maps $\chi_{17,4}^{(IV)^+}, 
\chi_{13,4}^{(II)^+}, \chi_{9,2}^{(II)^+}, \chi_{9,2}^{(II)^-}, \chi_{7,3}$  are not generically finite.
\end{corollary}

\begin{proof}
  This will follow  from  Lemmas \ref{lemma:tutteclassic}, \ref{lemma:ext7-nuovo},  \ref{lemma:ext7-nuovissimo}  and Proposition \ref{lemma:prok4div-nuovo} 
 below,  where  we  prove  that  the general  members of  $\E_{17,4}^{(IV)^+}$, 
$\E_{13,4}^{(II)^+}, \E_{9,2}^{(II)^+}, \E_{9,2}^{(II)^-}, \E_{7,3}$ are extendable.
\end{proof}

We will make use of the following auxiliary result:

\begin{lemma} \label{lemma:proj} 
 Let  $(S,L)$ be a polarized Enriques surface of genus $g \geq 6$ with $\phi(L) \geq 2$.  Assume that $(S,L+D)$ is extendable 
to an Enriques--Fano threefold $(Y,\H)$ for an effective divisor $D$, and that $Y$ is unirational. Then $(S,L)$ is extendable to an Enriques--Fano threefold 
$(X,\L)$ and the elements in $|\L|$ are in one-to one correspondence with the elements in $|\H \* \I_D|$.
 \end{lemma}

\begin{proof} 
 Let $\pi:\widetilde{Y} \to Y$ be a desingularization and identify $S$ with $\pi^{-1}(S)$. Then $h^1(\O_{\widetilde{Y}})=0$ by Lemma \ref{lemma:basicEF}(ii). Therefore, the exact sequence 
\begin{equation} \label{eq:exproj}
\xymatrix{ 0 \ar[r] & \O_{\widetilde{Y}} \ar[r] &  \pi^*\H \* \I_D  \ar[r] &  \O_S(L) \ar[r] &  0,}
\end{equation}
shows, as $|L|$ is  base point free and birational by \cite[Thms. 4.1 and 5.1]{cos1} or \cite[Thm. 4.4.1 and Prop. 4.7.1]{cd}, 
that the closure of the image of the rational map defined by the linear system $|\pi^*\H \* \I_D|$ is a threefold $X'$ in $\PP^g$, where $L^2=2g-2$, having the surfaces in $|\pi^*\H \* \I_D|$, including
$S$, as hyperplane sections. Since $Y$ is unirational, also $X'$ is. If $X'$ were a cone, then it would be birational to $H\times \PP^1$, for a general hyperplane section $H$ of $X'$. Thus, $H$ would be unirational, a contradiction. Hence, $X'$ is not a cone. Let $\nu: X \to X'$ be its normalization and $\L:=\nu^*\O_{X'}(1)$. Then $(X,\L)$ is an Enriques--Fano threefold extending $(S,L)$.
Identifying $D$ with $\pi^{-1}(D)$, we get, as $Y$ is normal,  
\[ H^0(Y,\H \* \I_D) \cong H^0({\widetilde{Y}},\pi^*\H \* \I_D) \cong H^0(X',\O_{X'}(1)) \cong \CC^{g+1},\]
and the latter is contained in $H^0(X',\nu_*\L) \cong H^0(X,\L)$. Since
$h^0(\L)=g+1$ by Lemma \ref{lemma:basicEF}, we must have 
$H^0(Y,\H \* \I_D) \cong H^0(X,\L)$, proving the last assertion.
\end{proof}

The \emph{classical Enriques--Fano threefold} $Y$ of genus $13$ is the
image of $\PP^ 3$ via the linear system of sextic surfaces that are
double along the edges of a tetrahedron, cf.\ \cite{CoMu,Fa}. 
Its smooth hyperplane sections are  Enriques surfaces with
polarization of
the form $2(E_1+E_2+E_3)$ (cf.\ \cite[Pf. of Prop. 13.1]{KLM}), that
is, they belong to $\E_{13,4}^{(II)^+}$.

\begin{lemma} \label{lemma:tutteclassic}
  Any $(S,H=2(E_1+E_2+E_3)) \in \E_{13,4}^{(II)^+}$ such that $E_1,E_2,E_3$ are nef and $|E_1+E_2+E_3|$ is birational is extendable to the classical Enriques--Fano threefold.
\end{lemma}

\begin{proof}
  By assumption, $|E_1+E_2+E_3|$ maps $S$ birationally onto a sextic surface in $\PP^3$ singular along the edges of a tetrahedron, which are the images of all $E_i$ and $E_i'$, the only member of $|E_i+K_S|$, for $i=1,2,3$, cf., e.g., \cite[Thm. 4.9.3]{cd}.  All such sextics are by construction hyperplane sections of the classical Enriques--Fano threefold. 
\end{proof}

\begin{lemma} \label{lemma:ext7-nuovo}
   A general member of  $\E_{7,3}$  is extendable. 
\end{lemma}

 \begin{proof} 
Let $(S, H) \in \E_{7,3}$ be general with
$H \sim E_1+E_2+E_3+E_4$.  In particular, $S$ is unnodal, whence
 $|E_1+E_2+E_3|$ is birational by \cite[Thm. 7.2]{cos1}.  Thus 
$(S,L:=2(E_1+E_2+E_3))$ is extendable to  the classical Enriques--Fano threefold  $Y$  by  Lemma \ref{lemma:tutteclassic}.  Note that $(E_1+E_2+E_3-E_4)^2=0$, so that $E_1+E_2+E_3 \sim E_4+F$, for an effective isotropic  $F$.  In
  particular, $L \sim H+F$, and the result follows from Lemma \ref{lemma:proj}.  
\end{proof}

Next we consider the  only known Enriques--Fano  threefold   of  genus $17$, namely the one constructed by Prokhorov in \cite[\S 3]{Pr} with canonical  nonterminal  singularities  in the following way: Let $x$ and $y_{i,j}$, $0 \leq i,j \leq 2$ be homogeneous coordinates in $\PP^9$ and consider the anticanonical embedding of $P:=\PP^1 \x \PP^1$ in $\PP^8=\{x=0\} \subset \PP^9$ given by
\[ (u_0:u_1) \x (v_0:v_1) \mapsto (y_{0,0}: \cdots: y_{2,2}), \; y_{i,j}=u_0^iu_1^{2-i}v_0^jv_1^{2-j}.\]
Let $V$ be the projective cone over $P$ and $v=(0:\cdots:0:1)$ its vertex. Then 
$V$ is a Gorenstein Fano threefold $V$ with canonical singularities.
 Let $\pi:V \to W$ be the quotient map of  the involution $\tau$ defined by
$\tau(x)=-x$ and $\tau(y_{i,j})=(-1)^{i+j}y_{i,j}$. Letting $\M:=\O_V(1)$, we have $-K_V \sim 2\M$ by \cite[Lemma 3.1]{Pr} and every smooth member of $|-K_V|$ is a $K3$ surface. The $\tau$-invariant ones are precisely the ones cut out on $V$ by quadrics of the form $q_1(y_{0,0},y_{0,2},y_{2,0},y_{2,2},y_{1,1})+q_2(y_{0,1},y_{2,1},y_{1,0},y_{1,2},x)$, where $q_1$ and $q_2$ are quadratic homogeneous forms, on which the action of $\tau$ is free. The quotient of any such  $\tau$-invariant  $\widetilde{S}$ by $\tau$ is thus an Enriques surface $S$. Since $\pi^*S = \widetilde{S}$ we have $2g-2=S^3=\frac{1}{2}\widetilde{S}^3=\frac{1}{2}(2\M)^3 =32$, whence $g=17$.
 Set $\L:=\O_W(S)$. Then $(W,\L)$ is an Enriques--Fano threefold of genus $17$.

 \begin{proposition} \label{lemma:prok4div-nuovo}
    The threefold $W$ is unirational and its  polarized Enriques sections 
 $(S,\L|_S)$  belong to  $\E_{17,4}^{(IV)^+}$.  Conversely, any $(S,H=4(E_1+E_2)) \in \E_{17,4}^{(IV)^+}$ with $E_1$ and $E_2$ nef is extendable to $(W, \L)$.
\end{proposition}

\begin{proof}
We keep the notation above. The unirationality of $W$ follows from the rationality of $V$.   Set $\L|_S=L$.   As  we have an induced double cover $\widetilde{S} \to \PP^1 \x \PP^1$, we have $\M|_{\widetilde{S}} \sim 2D$, with $D^2=4$. 
Thus, 
 $\pi|_{\widetilde{S}}^*L= \pi^*\L|_{\widetilde{S}} \sim \O_{\widetilde{S}}(\widetilde{S}) \sim 2\M|_{\widetilde{S}} \sim 4D$, so that  either $L$ or $L+K_S$ is $4$-divisible in $\Pic (S)$. By \cite[Prop. 12.1]{KLM}, either $L$ or $L-E$ with $E \cdot L =\phi(L)$ is $2$-divisible in $\Pic (S)$,
 and this implies that $L+K_S$ is not $2$-divisible. Hence,  $L$ is $4$-divisible in $\Pic (S)$, and the only possibility is $L \sim 4(E_1+E_2)$ as desired. 

 Now let $(S,H=4(E_1+E_2)) \in \E_{17,4}^{(IV)^+}$ with $E_1$ and $E_2$ nef, and denote by $p_i:S \to \PP^1$ the morphism induced by the pencil $|2E_i|$, $i=1,2$. Let $\pi: \widetilde{S} \to S$ be the $K3$ double cover. Then  
each $|\pi^*E_i|$ is an elliptic pencil and we 
have a commutative diagram
\[ \xymatrix{
\widetilde{S} \ar[r]^{\pi} \ar[d]_{\widetilde{p}_i} &  S \ar[d]^{p_i} \\
\PP^1 \ar[r]^{\sigma_i}  & \PP^1, }
\]
 where $\widetilde{p}_i$ is the map induced by the pencil $|\pi^*E_i|$ and 
$\sigma_i$ is  a double cover branched at the two points corresponding to the double fibers of $p_i$. 

The  map $\widetilde{S}\to \PP^1\times \PP^1$ given by $x \mapsto
(\widetilde{p}_1(x),\widetilde{p}_2(x))$ is a double cover  branched on a smooth curve in $|-2K_{\PP^1\times \PP^1}|$.  Equivalently, it is defined by the linear system $|\pi^*(E_1+E_2)|$, as its
image in $\PP^3$ factors through $\PP^1\times \PP^1$ by \cite{SD} 
(see also \cite[VIII.\S 18]{BPHV}).
 Then $\widetilde{S}$ is embedded in the total space $T(-K_{\PP^1\times \PP^1})$ of the line bundle $-K_{\PP^1\times \PP^1}$ on $\PP^1\times \PP^1$. The variety $T(-K_{\PP^1\times \PP^1})$ compactifies to $V'=\PP(-K_{\PP^1\times \PP^1}\oplus \O_{\PP^1\times \PP^1})$ by adding a section at infinity $\Sigma$ of $V'$ corresponding to the surjection $-K_{\PP^1\times \PP^1}\oplus \O_{\PP^1\times \PP^1}\to  \O_{\PP^1\times \PP^1}$. Then $V'$ identifies with the blow--up of the cone $V$ at its vertex, the exceptional divisor being $\Sigma$. We have the inclusion
$\widetilde{S}\subset T(-K_{\PP^1\times \PP^1})=V'-\Sigma\subset V'$, hence an inclusion $\widetilde{S}\subset V$, and it is easy to check that $\widetilde{S}$ identifies with a quadric section of $V$. 

Let now $t_i:\PP^1\to \PP^1$ be the involution corresponding to the double cover $\sigma_i$, for $i=1,2$. Consider the involution $t:(x,y)\in \PP^1\times \PP^1\to (t_1(x),t_2(y))\in \PP^1\times \PP^1$. By appropriately choosing coordinates we may assume that $t$ coincides with the involution  $t:(u_0:u_1) \x (v_0:v_1) \mapsto (-u_0:u_1) \x (-v_0:v_1)$. The involution $t$ on $\PP^1\times \PP^1$ lifts to the involution $\tau$ of $V$ defined above and one checks that $\widetilde{S}$ is $\tau$-invariant. Indeed, $\tau$ lifts to an involution of $\widetilde{S}$ because $t$ clearly fixes the branch divisor of the double cover $\widetilde{S}\to \PP^1\times \PP^1$, and $\widetilde{S}$ is also invariant by the involution of $T(-K_{\PP^1\times \PP^1})$ that sends a point $z$ in a fibre $\CC$ over a point of $\PP^1\times \PP^1$ to $-z$. 
Thus, the quotient map $\pi:V \to W$ maps $\widetilde{S}$ back to $S$, and the last assertion follows. 
\end{proof}

\begin{lemma} \label{lemma:ext7-nuovissimo}
  The general members of $\E_{9,2}^{(II)^+}$ and $\E_{9,2}^{(II)^-}$ are extendable.
\end{lemma}

\begin{proof}
Let $(S, H) \in \E_{9,2}^{(II)^+}$ (respectively, $\E_{9,2}^{(II)^-}$) be general. We have
$H \sim 4E_1+2E_2$ (resp., $4E_1+2E_2+K_S$).  
 Set   $D:=2E_2$ (resp., $2E_2+K_S$).
Then $(S,H+D\sim 4(E_1+E_2))$ is extendable to $(W,\L)$ by  Proposition  \ref{lemma:prok4div-nuovo}, and the result follows from Lemma \ref{lemma:proj}.
\end{proof}

\begin{remark} \label{rem:altricasi}
 Similar reasonings yield that the general members of $\E_{13,3}^{(II)}$, $\E_{10,3}^{(II)}$, $\E_{9,3}^{(II)}$, $\E_{7,2}^{(II)}$, $\E_{6,2}$ are extendable, and the corresponding moduli maps  are  not generically finite, but we will not need this fact. Similarly, a thorough study of the only Enriques--Fano threefold of genus $9$ in \cite{bay,sa} shows that the general member of $\E_{9,4}^{+}$ is extendable. 
 In particular, the general members of all the moduli spaces occurring in Corollary \ref{cor:EF-emb} are extendable to an Enriques--Fano threefold
$(X,\L)$ such that the morphism $\varphi_{\L}$ defined by $|\L|$ is an isomorphism on the general member $S$ of $|\L|$ (as $\phi(\L|_S) \geq 3$), whence $\varphi_{\L}(X) \subset \PP^g$ is not a cone and has smooth Enriques surfaces as hyperplane sections. 
\end{remark}

 We conclude the section by explaining how to use our results (without using \cite{KLM,Pr}) to  
bound the families of Enriques--Fano threefolds having the property that their Enriques sections are {\it general in moduli}, meaning that the family of polarized Enriques sections obtained from the family dominates the moduli space $\E$ of Enriques surfaces.

\begin{corollary} \label{cor:EF}
Consider a family of Enriques--Fano threefolds $(X,\L)$ such that $\L$ is globally generated and whose Enriques sections are general in moduli. Then the general polarized Enriques sections of the family belong to one of the following moduli spaces: 
\[ \E_{17,4}^{(IV)^+}, \; \E_{13,4}^{(II)^+}, \; \E_{13,3}^{(II)}, \; \E_{10,3}^{(II)}, \; \E_{9,4}^{+}, \; \E_{9,3}^{(II)}, \;  \E_{7,3},\; \E_{9,2}^{(II)^+}, \; \E_{9,2}^{(II)^-}, \] 
\[ \E_{7,2}^{(I)}, \; \E_{7,2}^{(II)}, \; \E_{6,2},  \; \E_{5,2}^{(I)}, \; \E_{5,2}^{(II)^+}, \; \E_{5,2}^{(II)^-}, \; \E_{4,2}, \; \E_{3,2}
\] 
\end{corollary}

\begin{proof}[Proof (granting Theorems \ref{mainthm1}-\ref{mainthm2})]
  Let $(X,\L)$ be general in the family and $S \in |\L|$ be general. The assumption that $\L$ is globally generated yields 
$\phi(\L|_S) \geq 2$ by 
\cite[Thms. 4.1]{cos1} or \cite[Prop. 4.7.1]{cd}. If $g \leq 5$, then $\phi \leq 2$ and 
$(S,\L|_S)$ belongs to one of $\E_{5,2}^{(I)}, \E_{5,2}^{(II)^+}, \E_{5,2}^{(II)^-}, \E_{4,2}, \E_{3,2}$, as those are all irreducible components of $\E_{g,2}$. If $g \geq 6$, then Proposition \ref{prop:pencil} and the assumptions of the corollary yield that $(S,\L|_S)$ belongs to one of the components of $\E_{g,\phi}$ over which 
the moduli map $\chi_{g,\phi}$ is not generically finite. These are given in 
Theorems \ref{mainthm1}-\ref{mainthm2}.
\end{proof}

\section{Computing cohomology of twisted tangent bundles} \label{sec:tools}

In the rest of the paper we adopt the following:

\begin{notation} \label{not:DC}
  For an Enriques surface $S$, we denote by $\pi:\widetilde{S} \to S$ the $K3$ double cover. For any divisor (or line bundle) $D$ on $S$ we write $\widetilde{D}:=\pi^*D$.
\end{notation}

In view of Corollary \ref{cor:geninj} and \eqref{eq:h1soprasotto}, in
this section we will develop some tools for computing or bounding
$h^1(\T_{\widetilde{S}}(-\widetilde{H}))$, where $H$ is a big and nef
line bundle on $S$.

Let  $F_1$ and $F_2$  be  two half-fibers such that $F_1 \cdot F_2=1$. Then 
$|\widetilde{F}_1+\widetilde{F}_2|$ is base point free  and (as in the proof of   Proposition  \ref{lemma:prok4div-nuovo})  it  defines a double cover $g:\widetilde{S} \to \PP^1 \x \PP^1$, branched on a smooth curve $R \in |-2K_{\PP^1 \x \PP^1}|$  (see also   \cite{SD} and  \cite[VIII.\S 18]{BPHV}). Denote by $\widetilde{R} \in |2\widetilde{F}_1+2\widetilde{F}_2|$ the ramification divisor. Define, for any big and nef $H$ on $S$

\[ \alpha=\alpha(H,F_1,F_2):=h^1(H-2F_1)+h^1(H-2F_1+K_S)+h^1(H-2F_2)+h^1(H-2F_2+K_S)\]
and 
\[\beta=\beta(H,F_1,F_2):=h^0(\O_{\widetilde{R}}(4\widetilde{F}_1+4\widetilde{F}_2
-\widetilde{H})).\]

\begin{lemma} \label{lemma:tricknuovo}
  Let $H$ be a big and nef line bundle on $S$ and $F_1$ and $F_2$ two half-fibers such that $F_1 \cdot F_2=1$. Then
$ h^1(\T_{\widetilde{S}}(-\widetilde{H})) \leq \alpha+\beta$,
with equality if $\alpha=0$.
\end{lemma}

\begin{proof}
  Dualizing the sequence of relative differentials  and tensoring 
by $\O_{\widetilde{S}}(-\widetilde{H})$ we get 
\[ 0 \to \T_{\widetilde{S}} (-\widetilde{H}) \to  g^*\T_{\PP^1 \x \PP^1}(-\widetilde{H})  \cong  \O_{\widetilde{S}}(2\widetilde{F}_1-\widetilde{H}) \+ \O_{\widetilde{S}}(2\widetilde{F}_2-\widetilde{H}) \to \O_{\widetilde{R}}(2\widetilde{R}-\widetilde{H}) \to 0.\]
Since $H$ is big and nef, we have $h^0(2\widetilde{F}_i-\widetilde{H})=0$, and the result follows. 
\end{proof}

The following bounds on $\beta$ will be useful later on:
\begin{equation}
  \label{eq:bbeta1}
  \beta=0 \quad \text{if} \quad (F_1+F_2)  \cdot H >8,
\end{equation}
and
\begin{multline}
\label{eq:bbeta2}
\qquad \beta \leq 
h^0(4F_1+4F_2-H)+h^0(4F_1+4F_2-H+K_S) \\
+ h^1(H-2F_1-2F_2)+h^1(H-2F_1-2F_2+K_S).\qquad
\end{multline}
Indeed, \eqref{eq:bbeta1} follows by reasons of degree, as
\begin{eqnarray*}
 \deg (\O_{\widetilde{R}}(4\widetilde{F}_1+4\widetilde{F}_2
-\widetilde{H}) ) & = & \pi^*(2F_1+2F_2) \cdot \pi^*(4F_1+4F_2-H) \\  
& = & 4(F_1+F_2)(4F_1+4F_2-H)=4\left(8-(F_1+F_2)  \cdot H\right),
\end{eqnarray*}
whereas \eqref{eq:bbeta2}  follows from the exact sequence
\begin{equation} \label{eq:bbeta3}
0 \longrightarrow \pi^*\O_S(2F_1+2F_2-H) \longrightarrow \pi^*\O_S(4F_1+4F_2-H)
\longrightarrow \O_{\widetilde{R}}(2\widetilde{R}-\widetilde{H}) \longrightarrow 0. 
\end{equation}

Let next $G_1$ and $G_2$  be  two  effective primitive  isotropic divisors such that $G_1 \cdot G_2=2$ and 
$G_1+G_2$ is nef (e.g., 
$G_1$ and $G_2$ are
half-fibers). Then  
$|\widetilde{G}_1+\widetilde{G}_2|$ is base point free and 
 embeds $\widetilde{S}$ into $\PP^5$ as   a complete intersection of three quadrics by \cite{SD}. Set

 \begin{eqnarray*}
 \gamma =\gamma(H,G_1,G_2)  & := & h^1(H-G_1-G_2)+h^1(H-G_1-G_2+K_S) \\
 \delta=\delta( H,G_1,G_2)  & := & h^0(2G_1+2G_2-H)+h^0(2G_1+2G_2-H+K_S) \\
 \epsilon=\epsilon( H,G_1,G_2) & := & \cork \mu_{\widetilde{G}_1+\widetilde{G}_2,\widetilde{H}-\widetilde{G}_1-\widetilde{G}_2},
 \end{eqnarray*}
 where
$\mu_{A, B}: H^0(A) \* H^0(B) \longrightarrow H^0(A+B)$
is the multiplication map of sections.

\begin{lemma} \label{lemma:trick4}
  Let $H$ be a big and nef line bundle on $S$ with $H^2 \geq 4$  and let $G_1$ and $G_2$ be two effective   primitive  isotropic  divisors such that $G_1 \cdot G_2=2$ and $G_1+G_2$ is nef. If $H \not \equiv G_1+G_2$, then
$h^1(\T_{\widetilde{S}}(-\widetilde{H}))  \leq \epsilon+6\gamma+3\delta$,
with equality if $\epsilon=\gamma=0$.
If $H \equiv G_1+G_2$, then $h^1(\T_{\widetilde{S}}(-\widetilde{H}))=12$.
\end{lemma}

\begin{proof}
  The Euler sequence 
twisted by $\O_{\widetilde{S}}(-\widetilde{H})$ is
\begin{equation*} \label{eq:euler}
\xymatrix{ 0 \ar[r] & \O_{\widetilde{S}}(-\widetilde{H}) \ar[r] & H^0(\widetilde{G}_1+\widetilde{G}_2)^{\vee} \* \O_{\widetilde{S}}(\widetilde{G}_1+\widetilde{G}_2-\widetilde{H}) \ar[r] & \T_{\PP^5}|_{\widetilde{S}}(-\widetilde{H})  \ar[r] & 0.}
\end{equation*}
The map on cohomology $H^2(\O_{\widetilde{S}}(-\widetilde{H})) \to H^0(\widetilde{G}_1+\widetilde{G}_2)^{\vee} \* H^2(\O_{\widetilde{S}}(\widetilde{G}_1+\widetilde{G}_2-\widetilde{H}))$ is the dual of $\mu_{\widetilde{G}_1+\widetilde{G}_2,\widetilde{H}-\widetilde{G}_1-\widetilde{G}_2}$. Thus, as $h^0(-\widetilde{H})=h^1(-\widetilde{H})=0$, we have
\begin{equation}
  \label{eq:h0t}
  h^0(\T_{\PP^5}|_{\widetilde{S}}(-\widetilde{H}))=6h^0(\O_{\widetilde{S}}(\widetilde{G}_1+\widetilde{G}_2-\widetilde{H}))=\begin{cases}
6 & \mbox{if} \; H \equiv G_1+G_2,\\
0 & \mbox{if} \; H \not \equiv G_1+G_2 
\end{cases}
\end{equation}
and
\begin{equation}
  \label{eq:h1t}
  h^1(\T_{\PP^5}|_{\widetilde{S}}(-\widetilde{H}))=\epsilon+6\gamma.
\end{equation}

 The  normal bundle  sequence 
twisted by $\O_{\widetilde{S}}(-\widetilde{H})$ is
\begin{equation} \label{eq:normal}
\xymatrix{ 0 \ar[r] & \T_{\widetilde{S}}(-\widetilde{H}) \ar[r] & \T_{\PP^5}|_{\widetilde{S}}(-\widetilde{H})  \ar[r] & 
\O_{\widetilde{S}}(2\widetilde{G}_1+2\widetilde{G}_2-\widetilde{H})^{\+ 3} \ar[r] & 0.}
\end{equation}
Taking cohomology, and using \eqref{eq:h0t} and \eqref{eq:h1t}, yields the desired result (using that $h^0(\T_{\widetilde{S}}(-\widetilde{H}))=0$ by the Mori-Sumihiro-Wahl Theorem \cite{MS,Wa-Pn}  when   $H \equiv G_1+G_2$).
\end{proof}

\begin{remark} \label{rem:push}
  Pushing forward \eqref{eq:normal} by $\pi$ and using  the fact that
$\T_{\widetilde{S}}\cong \pi^*\T_S $,   we obtain a  splitting of the coboundary map into the direct sum of
 $H^0(\O_S(2G_1+2G_2-H))^{\+ 3} \to H^1(\T_S(-H))$  and
 $H^0(\O_S(2G_1+2G_2-H+K_S))^{\+ 3} \to H^1(\T_S(-H+K_S))$.

\end{remark}

We end this section with some results that will be useful to compute the corank of multiplication maps. They are similar to the generalization by Mumford of a theorem of Castelnuovo, cf.\  \cite[Thm. 2, p.~41]{Mum}:

\begin{lemma} \label{lemma:mumford2}
 Let $F$ and $G$ be divisors on a projective surface  $S$  and assume that $|G|$ is a base point free pencil.
Then $\cork (\mu_{F,G}) \leq h^1(F-G)$, with equality if $h^1(F)=0$.
\end{lemma}

\begin{proof}
  This follows by tensoring the evaluation exact sequence
\[ 0 \longrightarrow \O_S(-G) \longrightarrow H^0(G) \* \O_S \longrightarrow  \O_S(G)  
\longrightarrow 0\]
with $\O_S(F)$ and taking cohomology. 
\end{proof}

\begin{lemma} \label{lemma:mumford3}
 \rm{(\cite[Obs.~1.4.1]{GP})}   {\it Let $F$ and $G=\sum_{i=1}^n G_i$ be divisors on a projective surface. If the multiplication maps $\mu_{F+G_1+\cdots+G_{i-1},G_i}$ are surjective for all $1 \leq i\leq n$, then also $\mu_{F,G}$ is surjective.}
\end{lemma}

\begin{remark} \label{rem:mumford3}
  When all $|G_i|$ are base point free pencils, the criterion in Lemma  \ref {lemma:mumford3} is satisfied (by Lemma \ref{lemma:mumford2}) if 
 $h^1(F+G_1+\cdots+G_{i-1}-G_i)=0$ for all $i=1,\ldots,n$. 
\end{remark}

\section{Fiber dimensions of moduli maps}
\label{sec:enr}

In this section we will apply the results of the previous section,
combined with Corollary \ref{cor:geninj} and \eqref{eq:h1soprasotto},
to prove Theorem \ref{mainthm1} and part of Theorem \ref{mainthm2}.

We will make use of the following facts.
Let $S$ be an Enriques surface.  By  \cite[Lemma~2.1]{klvan}, if $A$ and $B$
are effective 
divisors on $S$, then 
\begin{equation}
\label{kl-HIT}
A^2, B^2 \geq 0  \implies A \cdot B \geq 0,  \; \text{with equality iff} \;
A^2=0 \text{ and }
A \equiv k B \text{ for some } k\in\QQ.
\end{equation}
 For any divisor $D$ such that $D^2
\geq 0$  and $D \not \sim K_S$, either $D$ or $-D$ is effective. If moreover $S$ is unnodal, then any effective divisor $D$ is nef,
and it is ample if and only if $D^2>0$. Thus, for any divisor $D$ on  an unnodal  
$S$ we have  (by Riemann-Roch  and Mumford vanishing)  
\begin{equation}
\label{eq:noteff}
D^2 \leq -2
\ \implies\ 
\begin{cases}
h^0(D)=h^0(D+K_S)=h^2(D)=h^2(D+K_S)=0\quad \text{and} \\ 
\textstyle
h^1(D)=h^1(D+K_S)=-\frac{1}{2}D^2-1,
\end{cases}
\end{equation}
 \begin{equation}
\label{eq:KV}
h^1(D)=h^1(D+K_S)=0  \iff 
D^2 \geq -2 \;  
\text{and $D \not \equiv \pm lE$, $l \geq 2$, for  a half-fiber $E$}.
\end{equation}

 \begin{remark}  \label{rem:nodnotnec}   Recall that the general Enriques  surface is unnodal.  
The assumption  in the results below  that $S$ be unnodal  is not necessary  in order  to apply the results from 
\S \ref{sec:tools}. It  is  added to simplify the proofs of the vanishings of various cohomology groups. A more thorough study will yield bounds on $h^1(\T_S(-H))$ in terms of the existence of specific configurations of rational curves. We therefore expect that a result similar to Corollary \ref{cor:EF-emb} can also be obtained in the nodal cases (yielding for instance the two additional cases of Corollary \ref{cor:nodalbayle}   below), but it would  require additional  work that would bring us beyond the scope of this paper. 
 \end{remark}

We use Notations \ref{not:int} and \ref{not:DC}. 
We say that a simple isotropic decomposition
 $H \sim \sum_{i=1}^n \alpha_iF_i+ \varepsilon K_S$ contains 
$\sum_{i=1}^{n} a_iF_i$ if $\alpha_i \geq a_i$ for all $i \in \{1,\ldots,n\}$.

\begin{lemma} \label{lemma:appenriques}
Assume $S$ is an unnodal Enriques surface and $H$ a big and nef line
bundle on $S$.  We have $h^1(\T_{\widetilde{S}}(-\widetilde{H}))=0$
if a simple isotropic decomposition of $H$ contains \\[2mm]
\begin{minipage}{\linewidth/2}
\begin{itemize}
\item[(a)] $E_{1}+E_{2}+E_{3}+E_{4}+E_{5}$, 
\item[(b)] $2E_{1}+E_{2}+E_{3}+E_{4}$, 
\item[(c)] $3E_1+E_2+E_3$, 
\item[(d)] $5E_1+3E_2$,
\end{itemize}
\end{minipage}
\begin{minipage}{\linewidth/2}
\begin{itemize}
\item[(e)] $2E_1+E_3+E_{1,2}$,
\item[(f)] $E_1+E_2+E_{1,2}$,
\item[(g)] $3E_1+2E_{1,2}$,
\item[(h)] $2E_1+3E_{1,2}$.
\end{itemize}
\end{minipage}
\end{lemma}

\begin{proof}
\textbf{(a)} We have $H \sim E_{1}+E_{2}+E_{3}+E_{4}+E_{5}+D$, where $D$ is
 nef.   By \cite[Cor.~2.5.6]{cd} there are
primitive isotropic $F_1,F_2$ such that $F_1 \cdot F_2=F_i \cdot
E_j=1$ for $i \in \{1,2\}$ and $j \in \{1,2,3,4,5\}$ and such that
$F_j \not \equiv \frac{1}{2}(E_1+\cdots+E_{5})-\frac{1}{E_1 \cdot
D}D$.

We apply Lemma \ref{lemma:tricknuovo}.  
We have $(F_1+F_2) \cdot H
\geq 10$, whence $\beta=0$ by \eqref{eq:bbeta1}. We have
$(E_{1}+\cdots+E_{5}-2F_1)^2=0$, whence
\begin{eqnarray*} (H-2F_1)^2 & = & (E_{1}+\cdots+E_{5}-2F_1)^2+2(E_{1}+\cdots+E_{5}-2F_1)\cdot D + D^2 \\
& = & 2(E_{1}+\cdots+E_{5}-2F_1)\cdot D + D^2 \geq 0,
\end{eqnarray*}
with equality if and only if $D^2=0$ and $D \equiv
k(E_{1}+\cdots+E_{5}-2F_1)$ for some $k \in \QQ$
by \eqref{kl-HIT}.
In the latter case, intersecting with $E_1$ yields $k=\frac{1}{2}E_1
\cdot D$, whence $F_1 \equiv
\frac{1}{2}(E_1+\cdots+E_{5})-\frac{1}{E_1 \cdot D}D$, a
contradiction.  Hence $h^1(H-2F_1)=h^1(H-2F_1+K_S)=0$ by
\eqref{eq:KV}. By symmetry, also $h^1(H-2F_2)=h^1(H-2F_2+K_S)=0$, so
that $\alpha=0$.  The result then follows from Lemma
\ref{lemma:tricknuovo}.

\textbf{(b)} We have $H \sim 2E_1+E_{2}+E_{3}+E_{4}+D$, where $D$ is  nef.  By symmetry, we may assume that
\begin{equation}
  \label{eq:intD0}
  D \cdot E_4 \geq D \cdot E_2.
\end{equation}
We apply Lemma \ref{lemma:tricknuovo} with $F_1=E_1$ and $F_2=E_2$. 

We have $(H-2E_1)^2=(E_2+E_3+E_4+D)^2>0$, whence $h^1(H-2E_1) =
h^1(H-2E_1+K_S)=0$ by \eqref{eq:KV}. We have $(2E_1+E_3+E_4-E_2)^2=2$,
whence
\[
 (H-2E_2)^2  =  (2E_1+E_3+E_4-E_2)^2+2(2E_1+E_3+E_4-E_2)\cdot D + D^2 \geq 2,
\]
since both $2E_1+E_3+E_4-E_2$ and $D$ are effective.
 It follows that  $h^1(H-2E_2)=h^1(H-2E_2+K_S)=0$ again by \eqref{eq:KV}. Hence $\alpha=0$. 

We next prove that $\beta=0$. We will apply \eqref{eq:bbeta2}. We first note that,  by \eqref{eq:noteff},  
\[
h^0(4E_1+4E_2-H)=h^0(2E_1+3E_2-E_3-E_4-D) \leq h^0(2E_1+3E_2-E_3-E_4)=0,
\]
 as  $(2E_1+3E_2-E_3-E_4)^2=-6$. Similarly, $h^0(4E_1+4E_2-H+K_S)=0$. We also have
\begin{eqnarray*}
 (H-2E_1-2E_2)^2 & = & (E_3+E_4-E_2)^2+2(E_3+E_4-E_2)\cdot D + D^2 \\
& = &
-2 +2(E_3+E_4-E_2)\cdot D + D^2 \geq -2,
\end{eqnarray*} 
by \eqref{eq:intD0}, and $(H-2E_1-2E_2)^2=0$ if and only if $(D^2,
(E_3+E_4-E_2)\cdot D) \in \{(0,1),(2,0)\}$.  But in the latter case we
must have $D \cdot E_3=0$ by \eqref{eq:intD0}, contradicting
\eqref{kl-HIT}. In the former case we have $D \cdot
(H-2E_1-2E_2)=1$, implying that $H-2E_1-2E_2$ is primitive. Hence,
$h^1(H-2E_1-2E_2)=h^1(H-2E_1-2E_2+K_S)=0$ by \eqref{eq:KV}. Thus, $\beta=0$ by \eqref{eq:bbeta2}.

\textbf{(c)} We have $H \sim 3E_1+E_2+E_3+D$, where $D$ is  nef.  
By symmetry, we may assume that
\begin{equation}
  \label{eq:intD}
  D \cdot E_3 \geq D \cdot E_2.
\end{equation}
We apply Lemma \ref{lemma:tricknuovo} with $F_1=E_1$ and $F_2=E_2$. 

We have $(H-2E_1)^2=(E_1+E_2+E_3+D)^2>0$, whence $h^1(H-2E_1)  =  h^1(H-2E_1+K_S)=0$ by \eqref{eq:KV}. We have 
\begin{eqnarray*}
 (H-2E_2)^2 & = & (3E_1+E_3-E_2)^2+2(3E_1+E_3-E_2)\cdot D + D^2 \\
& = &
-2 +2(3E_1+E_3-E_2)\cdot D + D^2 \geq -2,
\end{eqnarray*}
by \eqref{eq:intD}, and $(H-2E_2)^2=0$ if and only if $(D^2,
(3E_1+E_3-E_2)\cdot D) \in \{(0,1),(2,0)\}$.  But in the latter case
we must have $D \cdot E_1=0$ by \eqref{eq:intD}, contradicting
\eqref{kl-HIT}. In the first case we have $D \cdot
(H-2E_2)=1$, implying that $H-2E_2$ is primitive. It follows that
$h^1(H-2E_2)=h^1(H-2E_2+K_S)=0$ again by \eqref{eq:KV}. Hence
$\alpha=0$.

We next prove that $\beta=0$. We will apply \eqref{eq:bbeta2}. We first note that,  by \eqref{eq:noteff},  
\[
h^0(4E_1+4E_2-H)=h^0(E_1+3E_2-E_3-D) \leq h^0(E_1+3E_2-E_3)=0,
\]
 as  $(E_1+3E_2-E_3)^2=-2$. Similarly, $h^0(4E_1+4E_2-H+K_S)=0$. We also have
\begin{eqnarray*}
 (H-2E_1-2E_2)^2 & = & (E_1+E_3-E_2)^2+2(E_1+E_3-E_2)\cdot D + D^2 \\
& = &
-2 +2(E_1+E_3-E_2)\cdot D + D^2 \geq -2,
\end{eqnarray*} 
by \eqref{eq:intD}, and $(H-2E_1-2E_2)^2=0$ if and only if $(D^2,
(E_1+E_3-E_2)\cdot D) \in \{(0,1),(2,0)\}$. But in the latter case we
must have $D \cdot E_1=0$ by \eqref{eq:intD}, contradicting
\eqref{kl-HIT}. In the first  we have $D \cdot
(H-2E_1-2E_2)=1$, implying that $H-2E_1-2E_2$ is primitive. Hence,
$h^1(H-2E_1-2E_2)=h^1(H-2E_1-2E_2+K_S)=0$ by \eqref{eq:KV}. Thus, $\beta=0$ by \eqref{eq:bbeta2}.

\textbf{(d)} We have $H \sim 5E_1+3E_2+D$, where $D$ is  nef.  
We apply Lemma \ref{lemma:tricknuovo} with $F_1=E_1$ and $F_2=E_2$ and argue as in (c). 

\textbf{(e)} We have $H \sim 2E_1+E_3+E_{1,2}+D$,  where $D$ is  nef.  We apply Lemma \ref{lemma:tricknuovo} with $F_1=E_1$ and $F_2=E_3$. 

We have 
$(H-2E_1)^2=(E_3+E_{1,2}+D)^2 >0$, whence $h^1(H-2E_1)=h^1(H-2E_1+K_S)=0$ by \eqref{eq:KV}. We have $(2E_1+E_{1,2}-E_3)^2 = 2$, whence  
\[ (H-2E_3)^2=(2E_1+E_{1,2}-E_3)^2+2(2E_1+E_{1,2}-E_3)\cdot D + D^2>0,\]
so that also $h^1(H-2E_3)=h^1(H-2E_3+K_S)=0$ by \eqref{eq:KV}. It follows that $\alpha=0$. 

To prove that $\beta=0$, we will use \eqref{eq:bbeta1} and \eqref{eq:bbeta2}. We first note that 
\begin{equation} \label{eq:mezzobeta}
h^0(4E_1+4E_3-H)=h^0(2E_1+3E_3-E_{1,2}-D) \leq h^0(2E_1+3E_3-E_{1,2})=0,
\end{equation}
by \eqref{eq:noteff}, as $(2E_1+3E_3-E_{1,2})^2=-2$. Similarly, $h^0(4E_1+4E_3-H+K_S)=0$. To finish the proof that $\beta=0$ we divide the treatment in different cases.

Assume that $E_{1,2}$ is present in the isotropic decomposition of $D$. Then
$E_1 \cdot H \geq 5$ and $E_3 \cdot H \geq 4$, so that $\beta=0$ by \eqref{eq:bbeta1}.

Assume that $E_3$ is present in the isotropic decomposition of $D$, whereas $E_{1,2}$ is not. Write $D'=D-E_3$. Then $H-2E_1-2E_3=E_{1,2}+D'$, so that 
$h^1(H-2E_1-2E_3)=h^1(H-2E_1-2E_3+K_S)=0$ by \eqref{eq:KV}. Hence $\beta=0$ by 
\eqref{eq:mezzobeta} and \eqref{eq:bbeta2}.

Assume that $E_j$ is present in the isotropic decomposition of $D$, for $j=1$ or $2$, whereas  $E_{1,2}$  is not. Then
$D \cdot (E_{1,2}-E_3) \geq 1$ and  
\begin{eqnarray*}
 (H-2E_1-2E_3)^2 & = & (-E_3+E_{1,2}+D)^2 =(E_{1,2}-E_3)^2+2(E_{1,2}-E_3)\cdot D +D^2 \\
& = & -2+2(E_{1,2}-E_3)\cdot D +D^2
 \geq 0,
\end{eqnarray*}
with equality if and only if $D=E_j$. In this  case,  $h^1(H-2E_1-2E_3)=h^1(H-2E_1-2E_3+K_S)=0$ by \eqref{eq:KV}, as $E_j \cdot(H-2E_1-2E_3)=1$.  Again, $\beta=0$ by 
\eqref{eq:mezzobeta} and \eqref{eq:bbeta2}.

Finally, assume that neither $E_1$, $E_2$, $E_3$ nor $E_{1,2}$ are
present in the isotropic decomposition of $D$. Then $D \cdot
(E_{1,2}-E_3)=0$, whence
\[ 
(H-2E_1-2E_3)^2=(-E_3+E_{1,2}+D)^2 =(E_{1,2}-E_3)^2+2(E_{1,2}-E_3)\cdot D +D^2
=D^2 -2
 \geq -2,
\]
and is $0$ if and only if $D^2=2$. In this case, we have $E_1 \cdot D
\geq 2$ and $E_3 \cdot D \geq 2$, 
so that $E_1 \cdot H \geq 5$ and
$E_3 \cdot H \geq 5$. Hence, $\beta=0$ by \eqref{eq:bbeta1}.

\textbf{(f)} We have $H \sim E_1+E_2+E_{1,2}+D$,  where $D$ is  nef. 
By symmetry between $E_1$ and $E_2$, we may assume that $D \not \equiv kE_2$ for any $k \geq 1$. 
We apply Lemma \ref{lemma:trick4} with $G_1=E_1$ and $G_2=E_{1,2}$.

We have $H-E_1-E_{1,2}=E_2+D$, whence 
$h^1(H-E_1-E_{1,2})=h^1(H-E_1-E_{1,2}+K_S)=0$ by \eqref{eq:KV} and the fact that $D \not \equiv kE_2$. Therefore, $\gamma=0$. 

 By \eqref{eq:noteff} and the fact that $(E_1+E_{1,2}-E_2)^2=-2$, we
 have 
\[ h^0(2E_1+2E_{1,2}-H)=h^0(E_1+E_{1,2}-E_2-D) \leq h^0(E_1+E_{1,2}-E_2)=0.\]
 Similarly,   
$h^0(2E_1+2E_{1,2}-H+K_S)=0$. Hence, $\delta=0$.

To check that the multiplication map $\mu_{\widetilde{E}_1+\widetilde{E}_{1,2},\widetilde{E}_2+\widetilde{D}}$ is surjective, we apply Lemmas \ref{lemma:mumford2} and \ref{lemma:mumford3}, cf.\  Remark \ref{rem:mumford3}. 
Write $D  \equiv \sum_{i=1}^{n}\alpha_i E_i + \alpha_0E_{1,2}$  for some $n \leq 9$.  The multiplication map
\[ \mu_{\widetilde{E}_1+\widetilde{E}_{1,2},\widetilde{E}_2}: H^0(\widetilde{E}_1+\widetilde{E}_{1,2}) \* H^0(\widetilde{E}_2) \longrightarrow H^0(\widetilde{E}_1+\widetilde{E}_{1,2}+\widetilde{E}_2) \]
is surjective, since   \eqref{eq:KV} and the fact that $(E_1+E_{1,2}-E_2)^2=-2$ imply that 
\[h^1(\widetilde{E}_1+\widetilde{E}_{1,2}-\widetilde{E}_2)=h^1(E_1+E_{1,2}-E_2)+h^1(E_1+E_{1,2}-E_2+K_S)=0.\]
 Likewise, all multiplication maps
$\mu_{\widetilde{E}_1+\widetilde{E}_{1,2}+j\widetilde{E}_2,\widetilde{E}_2}$
for $1 \leq j \leq \alpha_2$ are surjective, since
all
$\bigl((\widetilde{E}_1+\widetilde{E}_{1,2}+j\widetilde{E}_2)
-\widetilde{E}_2\bigr)^2 >0$. For the same reason, all
$\mu_{\widetilde{E}_1+\widetilde{E}_{1,2}+(\alpha_2+1)\widetilde{E}_2+j\widetilde{E}_1,\widetilde{E}_1}$,
for $0 \leq j \leq \alpha_1-1$, are surjective, as well as all
$\mu_{(\alpha_1+1)\widetilde{E}_1+\widetilde{E}_{1,2}+(\alpha_2+1)\widetilde{E}_2+j\widetilde{E}_{1,2},\widetilde{E}_2}$,
for $0 \leq j \leq \alpha_0-1$. Finally, for any $i \in
\{3,\ldots,n\}$ and any $0 \leq j \leq \alpha_i-1$,
 set 
\[
B_{ij}:=(\alpha_1+1)E_1+(\alpha_0+1)E_{1,2}+(\alpha_2+1)E_2+\alpha_3
E_3+\cdots + \alpha_{i-1}E_{i-1}+jE_{i}-E_{i}.\]
Then $\widetilde{B}_{ij}=\widetilde{E}_1+\widetilde{E}_{1,2}+
\widetilde{E}_2-\widetilde{E}_{i} +\Delta$,
 with
\[ \Delta:=
\alpha_1\widetilde{E}_1+\alpha_0\widetilde{E}_{1,2}+\alpha_2\widetilde{E}_2+\alpha_3
\widetilde{E}_3+\cdots +
\alpha_{i-1}\widetilde{E}_{i-1}+j\widetilde{E}_{i}.\]
Since $\Delta^2 \geq 0$ and 
$(\widetilde{E}_1+\widetilde{E}_{1,2}+
\widetilde{E}_2-\widetilde{E}_{i})^2= 8$,
we have  $\widetilde{B}_{ij}^2 > 0$,  
whence  $h^1(\widetilde{B}_{ij})=h^1(B_{ij})+h^1(B_{ij}+K_S)=0$ 
by \eqref{eq:KV}. It follows by Lemma \ref{lemma:mumford3} and  Remark \ref{rem:mumford3} that $\mu_{\widetilde{E}_1+\widetilde{E}_{1,2},\widetilde{E}_2+\widetilde{D}}$ is surjective, whence $\epsilon=0$.

\textbf{(g)} We have $H \sim 3E_1+2E_{1,2}+D$,  where $D$ is  nef.  
If $E_2$ is present in $D$, we are done by (f). If any $E_j$, for $j \neq 1,2$, is present in $D$, we are done by (e). We have therefore left to treat the case where $H \equiv a_1E_1+a_0E_{1,2}$, with $a_1 \geq 3$ and $a_0 \geq 2$. By symmetry, we may assume that $a_1 \geq a_0$. As in the previous case, we apply Lemma \ref{lemma:trick4} with $G_1=E_1$ and $G_2=E_{1,2}$.

We have $H-E_1-E_{1,2}=(a_1-1)E_1+(a_2-1)E_{1,2}$, whence 
$h^1(H-E_1-E_{1,2})=h^1(H-E_1-E_{1,2}+K_S)=0$ by \eqref{eq:KV}. Therefore, $\gamma=0$. 

We have $2E_1+2E_{1,2}-H\equiv -(a_1-2)E_1-(a_0-2)E_{1,2}$, whence 
$h^0(2E_1+2E_{1,2}-H)=h^0(2E_1+2E_{1,2}-H+K_S)=0$.  Thus,  $\delta=0$.

To check that the map
$\mu_{\widetilde{E}_1+\widetilde{E}_{1,2},(a_1-1)\widetilde{E}_1+(a_0-1)\widetilde{E}_{1,2}}$
is surjective, we apply Lemma \ref{lemma:mumford3}.  The map
$\mu_{\widetilde{E}_1+\widetilde{E}_{1,2},(a_0-1)\left(\widetilde{E}_1+\widetilde{E}_{1,2}\right)}$
is surjective by \cite[Thm.~6.1]{SD}. Finally, for $0 \leq j \leq
a_1-a_0-1$, set $B_j:=(a_0+j)E_1+(a_0-1)E_{1,2}-E_1$. Then $B_j^2 >0$,
so that $h^1(\widetilde{B}_j)=h^1(B)+h^1(B+K_S)=0$ by \eqref{eq:KV},
whence all $\mu_{(a_0+j)\widetilde{E}_1+(a_0-1)\widetilde{E}_{1,2},
\widetilde{E}_1}$ are surjective by Lemma \ref{lemma:mumford2}.
 The  map
$\mu_{\widetilde{E}_1+\widetilde{E}_{1,2},(a_1-1)\widetilde{E}_1+(a_0-1)\widetilde{E}_{1,2}}$
is  thus  surjective by Lemma \ref{lemma:mumford3}.

\textbf{(h)} This case is treated  as the previous one, exchanging the roles of $E_1$ and $E_{1,2}$.
\end{proof}

\begin{lemma} \label{lemma:restanti} 
Assume $S$ is an unnodal Enriques surface  and $H$ a big and nef line bundle on $S$.  If  $h^1(\T_{\widetilde{S}}(-\widetilde{H}))\neq 0$, then we are in one of the following cases: 
\vspace{0.2cm}

\begin{footnotesize}

\begin{tabular}{| l| l |l|}
\hline
\raisebox{1mm}{ component of moduli  space} & \raisebox{1mm}{simple isotr.\ decomp.} & 
\raisebox{1mm}[1.5em]
{$h^1(\T_{\widetilde{S}}(-\widetilde{H}))$} \\ \hline
 $\E_{17,4}^{(IV)^+}$ \rm{and} $\E_{17,4}^{(IV)^-}$ & $H \equiv 4(E_1+E_2)$ &  $=1$ \\
  $\E_{13,3}^{(II)}$ &$H \sim 4E_1+3E_2$ & $=2$ \\
 $\E_{13,4}^{(II)^+}$ \rm{and} $\E_{13,4}^{(II)^-}$ & $H \equiv 2(E_1+E_2+E_3)$  & $\leq 1$\\
 $\E_{10,3}^{(II)}$ & $H \sim 3E_1+3E_2$ & $=4$ \\
$\E_{9,3}^{(II)}$ &$H \sim 2E_1+2E_2+E_3$ &  $=2$ \\
$\E_{9,4}^{+}$ \rm{and} $\E_{9,4}^{-}$ & $H \equiv 2(E_1+E_{1,2})$ &  $=3$\\
 $\E_{7,3}$ & $H \sim E_1+E_2+E_3+E_4$ &  $\leq 2$ \\
 $\E_{6,2}$ &   $H \sim 2E_1+E_2+E_3$ & $ =4$ \\
 $\E_{4,2}$ & $H \sim E_1+E_2+E_3$ & $=8$ \\ 
$\E_{3,2}$ & $H \sim E_1+E_{1,2}$ & $=12$\\ 
 $\E_{2k+1,2}^{(I)}$, $k \geq 2$ & $H \sim kE_1+E_{1,2}$, $k \geq 2$ & \\ 
$\E_{2k+1,2}^{(II)}$ if $k$ is odd; 
$\E_{2k+1,2}^{(II)+}$ and $\E_{2k+1,2}^{(II)-}$ if $k$ is even & $H \equiv kE_1+2E_{2}$, $k \geq 2$ & \\
$\E_{k+1,1}$ & $H \equiv kE_1+E_{2}$, $k \geq 1$ & \\
\hline
\end{tabular}
\end{footnotesize}

\end{lemma}

\begin{proof}
Up to rearranging indices, the decompositions in the table are the only ones not covered by Lemma \ref{lemma:appenriques}, except for $H \equiv E_1+kE_3+lE_{1,2}$, with $k,l \geq 1$. Set $F:=E_1+E_{1,2}-E_3$. Then $F^2=0$, $E_{1,2} \cdot F=1$ and $E_3 \cdot F=2$. Thus, $H \equiv (k+1)E_3+F+(l-1)E_{1,2}$ is a simple isotropic decomposition, which can be rewritten, after renaming indices, $H \equiv (k+1)E_1+E_{1,2}+(l-1)E_3$.  This falls into case (e) of Lemma \ref{lemma:appenriques} if $l \geq 2$, and is present in the table of the lemma if $l=1$.  We now study $h^1(\T_{\widetilde{S}}(-\widetilde{H}))$.

 $\bullet$ {\bf  Cases  $\E_{17,4}^{(IV)^+}$ and $\E_{17,4}^{(IV)^-}$.}   We apply Lemma \ref{lemma:tricknuovo} with $F_1=E_1$ and $F_2=E_2$.
We have 
$(H-2E_1)^2=(2E_1+4E_{2})^2 >0$, whence $h^1(H-2E_1)=h^1(H-2E_1+K_S)=0$ by \eqref{eq:KV}. Similarly, we have $h^1(H-2E_2)=h^1(H-2E_2+K_S)=0$, so that  $\alpha=0$. 
We  have $4\widetilde{F}_1+4\widetilde{F}_2-\widetilde{H} =0$, whence $\beta=1$ by definition.  Lemma \ref{lemma:tricknuovo} implies 
 $h^1(\T_{\widetilde{S}}(-\widetilde{H}))=1$.

$\bullet$  {\bf Case $\E_{13,3}^{(II)}$.} We apply Lemma \ref{lemma:tricknuovo} with $F_1=E_1$ and $F_2=E_2$.
 We find $\alpha=0$ as   above.  We have $4F_1+4F_2-H=E_2$ and $2F_1+2F_2-H=-2E_1-E_2$. Hence $h^0(4\widetilde{F}_1+4\widetilde{F}_2-\widetilde{H})=h^0(E_2)+h^0(E_2+K_S)=2$, and $h^i(2\widetilde{F}_1+2\widetilde{F}_2-\widetilde{H})=0$ for $i=0,1$. Hence, $\beta=2$ by the exact sequence \eqref{eq:bbeta3}.  Thus, $h^1(\T_{\widetilde{S}}(-\widetilde{H}))=2$ by  Lemma \ref{lemma:tricknuovo}.

 $\bullet$ {\bf  Cases  $\E_{13,4}^{(II)^+}$ and $\E_{13,4}^{(II)^-}$.} We apply Lemma \ref{lemma:tricknuovo} with  $F_1=E_1$ and $F_2=E_2$. We find $\alpha=0$ as above.  We have $4F_1+4F_2-H \equiv  2(E_1+E_2-E_3)$, which has square $-8$, whence $h^0(4F_1+4F_2-H)=h^0(4F_1+4F_2-H+K_S)=0$ by \eqref{eq:noteff}. We have $2F_1+2F_2-H  \equiv -2E_3$, whence $h^1(2F_1+2F_2-H)+h^1(2F_1+2F_2-H+K_S)=1$.  Therefore,  $\beta \leq 1$ by \eqref{eq:bbeta2}.  It follows that $h^1(\T_{\widetilde{S}}(-\widetilde{H})) \leq 1$ by   Lemma \ref{lemma:tricknuovo}.

$\bullet$ {\bf Case $\E_{10,3}^{(II)}$.} We apply Lemma \ref{lemma:tricknuovo} with $F_1=E_1$ and $F_2=E_2$.  As above,  $\alpha=0$.  We have $2F_1+2F_2-H=-E_1-E_2$, whence $h^i(2F_1+2F_2-H)=h^i(2F_1+2F_2-H+K_S)=0$,  $i=0,1$ by \eqref{eq:KV}. We have $4F_1+4F_2-H=E_1+E_2$, whence $h^0(4F_1+4F_2-H)=h^0(4F_1+4F_2-H+K_S)=2$. 
Thus,
$h^i(2\widetilde{F}_1+2\widetilde{F}_2-\widetilde{H})=0$,  $i=0,1$
and $h^0(4\widetilde{F}_1+4\widetilde{F}_2-\widetilde{H})=4$,  so  $\beta=4$  by \eqref{eq:bbeta3}. Lemma \ref{lemma:tricknuovo} yields 
  $h^1(\T_{\widetilde{S}}(-\widetilde{H}))=4$.

$\bullet$ {\bf Case $\E_{9,3}^{(II)}$.} We apply Lemma \ref{lemma:tricknuovo} with
$F_1=E_1$ and $F_2=E_2$.  We find $\alpha=0$ as above. We have $2F_1+2F_2-H=-E_3$, whence $h^i(2F_1+2F_2-H)=h^i(2F_1+2F_2-H+K_S)=0$ for $i=0,1$ by \eqref{eq:KV}. We have $4F_1+4F_2-H=2E_1+2E_2-E_3$, which has square $0$. Since $E_1 \cdot (4F_1+4F_2-H)=1$, we have $h^0(4F_1+4F_2-H)=h^0(4F_1+4F_2-H+K_S)=1$ (using \eqref{eq:KV}). It follows that 
$h^i(2\widetilde{F}_1+2\widetilde{F}_2-\widetilde{H})=0$ for $i=0,1$
and $h^0(4\widetilde{F}_1+4\widetilde{F}_2-\widetilde{H})=2$, whence $\beta=2$ by \eqref{eq:bbeta3}.  Thus, $h^1(\T_{\widetilde{S}}(-\widetilde{H}))=2$ by  Lemma \ref{lemma:tricknuovo}.

 $\bullet$ {\bf  Cases   $\E_{9,4}^{+}$ and $\E_{9,4}^{-}$.} We apply Lemma \ref{lemma:trick4} with $G_1=E_1$ and $G_2=E_{1,2}$. 
We have $H-G_1-G_2  \equiv  E_1+E_{1,2}$, whence $\gamma=0$. We have $2G_1+2G_2-H  \equiv  0$, whence $\delta=1$. Finally, the multiplication map $\mu_{\widetilde{E}_1+ \widetilde{E}_{1,2},\widetilde{E}_1+ \widetilde{E}_{1,2}}$ is surjective by \cite[Thm. 6.1]{SD}. Hence, $\epsilon=0$.  Thus, $h^1(\T_{\widetilde{S}}(-\widetilde{H}))=3$ by  Lemma \ref{lemma:trick4}. 
 (See also Remark \ref{rem:push2} below.)

 $\bullet$ {\bf Case $\E_{7,3}$.} We apply Lemma \ref{lemma:tricknuovo} with 
$F_1=E_1$ and $F_2=E_2$. 
We have $(H-2F_1)^2=(E_2+E_3+E_4-E_1)^2=0$ and $E_2 \cdot(H-2F_1)=1$, whence
$h^1(H-2F_1)=h^1(H-2F_1+K_S)=0$ by \eqref{eq:KV}. Similarly, $h^1(H-2F_2)=h^1(H-2F_2+K_S)=0$, whence $\alpha=0$. We have $(4F_1+4F_2-H)^2=(3E_1+3E_2-E_3-E_4)^2=-4$, whence $h^0(4F_1+4F_2-H)=h^0(4F_1+4F_2-H+K_S)=0$ by \eqref{eq:noteff}. We have $(2F_1+2F_2-H)^2=(E_1+E_2-E_3-E_4)^2=-4$, whence $h^1(2F_1+2F_2-H)=h^1(2F_1+2F_2-H+K_S)=1$ by \eqref{eq:noteff}. 
It follows that 
$h^1(2\widetilde{F}_1+2\widetilde{F}_2-\widetilde{H})=2$ 
and $h^0(4\widetilde{F}_1+4\widetilde{F}_2-\widetilde{H})=0$, whence $\beta \leq 2$ by \eqref{eq:bbeta3}.  Thus, $h^1(\T_{\widetilde{S}}(-\widetilde{H})) \leq 2$ by  Lemma \ref{lemma:tricknuovo}.

$\bullet$ {\bf Case $\E_{6,2}$.} We apply Lemma \ref{lemma:tricknuovo} with $F_1=E_1$ and $F_2=E_2$.
We have $(H-2F_1)^2=(E_2+E_3)^2=2$, whence  $h^1(H-2F_1)=h^1(H-2F_1+K_S)=0$ by \eqref{eq:KV}. We have 
$(H-2F_2)^2=(2E_1+E_3-E_2)^2=-2$, whence
$h^1(H-2F_2)=h^1(H-2F_2+K_S)=0$ by \eqref{eq:KV}. It follows that $\alpha=0$. 
We have $(4F_1+4F_2-H)^2=(2E_1+3E_2-E_3)^2=2$, whence $h^0(4F_1+4F_2-H)=h^0(4F_1+4F_2-H+K_S)=2$ by \eqref{eq:KV} and Riemann-Roch. We have $(2F_1+2F_2-H)^2=(E_2-E_3)^2=-2$, whence $h^i(2F_1+2F_2-H)=h^i(2F_1+2F_2-H+K_S)=0$ for $i=0,1$ by \eqref{eq:noteff}. Thus,
$h^i(2\widetilde{F}_1+2\widetilde{F}_2-\widetilde{H})=0$  for $i=0,1$
and $h^0(4\widetilde{F}_1+4\widetilde{F}_2-\widetilde{H})=4$, so $\beta =4$ by \eqref{eq:bbeta3}. Lemma \ref{lemma:tricknuovo} yields  $h^1(\T_{\widetilde{S}}(-\widetilde{H})) =4$.

$\bullet$ {\bf Case $\E_{4,2}$.} We apply Lemma \ref{lemma:tricknuovo} with 
$F_1=E_1$ and $F_2=E_2$. 
We have $(H-2F_1)^2=(E_2+E_3-E_1)^2=-2$, whence
$h^1(H-2F_1)=h^1(H-2F_1+K_S)=0$ by \eqref{eq:KV}. Similarly, $h^1(H-2F_2)=h^1(H-2F_2+K_S)=0$, whence $\alpha=0$. We have $(4F_1+4F_2-H)^2=(3E_1+3E_2-E_3)^2=6$, whence $h^0(4F_1+4F_2-H)=h^0(4F_1+4F_2-H+K_S)=4$. We have $(2F_1+2F_2-H)^2=(E_1+E_2-E_3)^2=-2$, whence $h^i(2F_1+2F_2-H)=h^i(2F_1+2F_2-H+K_S)=0$,  $i=0,1$ by \eqref{eq:noteff}. Thus,
$h^i(2\widetilde{F}_1+2\widetilde{F}_2-\widetilde{H})=0$ for $i=0,1$ 
and $h^0(4\widetilde{F}_1+4\widetilde{F}_2-\widetilde{H})=8$, so $\beta =8$ by \eqref{eq:bbeta3}. Lemma \ref{lemma:tricknuovo} yields  $h^1(\T_{\widetilde{S}}(-\widetilde{H})) =8$. 

$\bullet$ {\bf Case $\E_{3,2}$.}  Lemma \ref{lemma:trick4} with $G_1=E_1$ and $G_2=E_{1,2}$ yields $h^1(\T_{\widetilde{S}}(-\widetilde{H})) =12$.
\end{proof}

\begin{remark} \label{rem:push2}
  In the cases $\E_{9,4}^{+}$ and $\E_{9,4}^{-}$, applying Remark \ref{rem:push}, we obtain more precisely that
$h^1(\T_S(-H))=3$ and $h^1(\T_S(-H+K_S))=0$ for $(S,H) \in \E_{9,4}^{+}$  and 
$h^1(\T_S(-H))=0$ and $h^1(\T_S(-H+K_S))=3$ for $(S,H) \in \E_{9,4}^{-}$.

\end{remark}

 We  draw some consequences from the last two lemmas:

\begin{proof}[Proof of Theorem \ref{mainthm1}]
The cases not in the table  of Lemma \ref{lemma:restanti} satisfy  $h^1(\T_{\widetilde{S}}(-\widetilde{H}))= 0$, where the result follows from Corollary \ref{cor:geninj}  and \eqref{eq:h1soprasotto}.  Let us  consider the  other  cases.

$\bullet$ {\bf  Cases    $\E_{17,4}^{(IV)^+}$ and
$\E_{17,4}^{(IV)^-}$.}   The moduli map $\chi_{17,4}^{(IV)^+}$ is not generically finite by Corollary \ref{cor:EF1}, whence $h^1(\T_S(-H))>0$  for $(S,H) \in \E_{17,4}^{(IV)^+}$  by Corollary
\ref{cor:geninj}.
Lemma \ref{lemma:restanti} then implies that $h^1(\T_S(-H))=1$ and
$h^1(\T_S(-H+K_S))=0$, so that $\chi_{17,4}^{(IV)^+}$ has generically
one-dimensional fibers and $\chi_{17,4}^{(IV)^-}$ is generically
finite by Corollary \ref{cor:geninj}.

$\bullet$ {\bf Cases $\E_{13,4}^{(II)^+}$ and $\E_{13,4}^{(II)^-}$.} 
These cases are treated exactly as the previous ones. 

$\bullet$ {\bf  Cases    $\E_{9,4}^{+}$ and $\E_{9,4}^{-}$.} Lemma \ref{lemma:restanti} and Remark \ref{rem:push2} imply that
for $(S,H) \in \E_{9,4}^{+}$, we have $h^1(\T_S(-H))=3$ and $h^1(\T_S(-H+K_S))=0$. Thus $\chi_{9,4}^{+}$ has generically three-dimensional fibers by Corollary \ref{cor:geninj}. It also follows that 
$h^1(\T_S(-H))=0$ for $(S,H) \in \E_{9,4}^{-}$, whence $\chi_{9,4}^{-}$ is generically finite. 

$\bullet$ {\bf Case $\E_{7,3}$.} 
By  Corollary \ref{cor:EF1},  the moduli map $\chi_{7,3}$ is not
generically finite, whence $h^1(\T_S(-H))>0$  for $(S,H) \in \E_{7,3}$  by Corollary
\ref{cor:geninj}, and also $h^1(\T_S(-H+K_S))>0$, since $(S,H+K_S) \in
\E_{7,3}$ as well.  Lemma \ref{lemma:restanti} then implies that
$h^1(\T_S(-H))=h^1(\T_S(-H+K_S))=1$, in particular $\chi_{7,3}$ has
generically one-dimensional fibers by Corollary \ref{cor:geninj}.

$\bullet$ {\bf  Cases   $\E_{13,3}^{(II)}$, $\E_{10,3}^{(II)}$ and $\E_{9,3}^{(II)}$.} Since these spaces are all irreducible and $(S,H)$ and $(S,H+K_S)$ belong to the same spaces, we must have $h^1(\T_S(-H))=h^1 (\T_S(-H+K_S))=\frac{1}{2}h^1(\T_{\widetilde{S}}(-\widetilde{H}))$. Then Lemma \ref{lemma:restanti} and 
 Corollary \ref{cor:geninj} yield the rest.
\end{proof}

\begin{proof}[Proof of Corollary \ref{cor:EF-emb}]
  Assume that $S \subset \PP^N$ is $k$-extendable, for some $k \geq 1$.  In the language of  \cite{Lv} this means that $S$ can be nontrivially extended $k$ steps. Since $S$ is not a quadric, \cite[Thm. 0.1]{Lv} yields that
  \begin{equation}
    \label{eq:lvov}
    \alpha(S):=h^0(\N_{S/\PP^N}(-1))-N-1 \geq \min\{k,N\}.
  \end{equation}

The normal bundle and Euler sequences  yield
$h^0(\N_{S/\PP^N}(-1)) \leq N+1+h^1(\T_S(-1))$, whence $\alpha(S) \leq h^1(\T_S(-1))$. 
Hence, $h^1(\T_S(-1)) \geq \min\{k,N\}$ by \eqref{eq:lvov}.  In particular, we must have $h^1(\T_S(-1))>0$, which may also be deduced from Lemma \ref{lemma:diff} and Proposition \ref{prop:pencil}. 
Since $\phi(\O_S(1)) \geq 3$ by \cite[Thm. 5.1]{cos1} or \cite[Thm. 4.6.1]{cd} and we assume  
$S$ is  unnodal,  $(S,\O_S(1))$  must therefore be in one of the cases listed in Lemma \ref{lemma:restanti}. The proof of Theorem \ref{mainthm1}
shows that $h^1(\T_S(-1))=0$ for $(S,\O_S(1))$ in $\E_{17,4}^{(IV)^-}$, $\E_{13,4}^{(II)^-}$ and $\E_{9,4}^{-}$, leaving us with the list of the corollary. The same proof also shows that $h^1(\T_S(-1))=1$ in all cases, except for the cases
$\E_{10,3}^{(II)}$ and  $\E_{9,4}^{+}$, where $h^1(\T_S(-1))=2$ and $3$, respectively. Since $N \geq 4$, we get from from  \eqref{eq:lvov} 
that $k \leq 2$, resp. $3$, in these cases. 
\end{proof}

\begin{corollary} \label{cor:nodalbayle}
  The  general  Enriques  surface  sections  of the Enriques--Fano threefolds (1) and (3) in  the list of \cite[Thm. A]{bay}  are nodal Enriques surfaces.
\end{corollary}

\begin{proof} Let  $(X,\L)$  be one of the Enriques--Fano threefolds in question  and $S \in |\L|$ be general.  We have $(g,\phi)=(8,3)$ and $(6,3)$, which do not appear in the table of Lemma \ref {lemma:restanti}.  By Proposition \ref  {prop:pencil}, the map $c_{g,\phi}$ has positive dimensional fiber at
$(S,\L|_S,C)$  for general $C \in |\L|_S|$. The result thus follows from Lemma \ref {lemma:restanti}, Corollary \ref{cor:geninj}  and \eqref{eq:h1soprasotto}.
\end{proof}

 The next result proves part of Theorem \ref{mainthm2}. 

 \begin{proposition} \label{prop:fi2gpari}
    (i) The  moduli map $\chi_{g,2}$ is generically finite  for even  $g \geq 8$, dominant  for  $g =3,4$, and with image of codimension 2  for   $g=6$.

   (ii) A general fiber of  $\chi^{(I)}_{5,2}$  is  three-dimensional.  
 \end{proposition}

 \begin{proof}
(i)   By  Lemma \ref{lemma:restanti}   we have $h^1(\T_{\widetilde{S}}(-\widetilde{H}))=0$  for even   $g \geq 8$, and the result follows from Corollary \ref{cor:geninj}  and \eqref{eq:h1soprasotto}.   
 In the remaining cases, as $(S,H)$ and $(S,H+K_S)$  both belong to $\E_{g,2}$, which is irreducible by \cite{cdgk},  we must have $h^1(\T_S(-H))=h^1 (\T_S(-H+K_S))=\frac{1}{2}h^1(\T_{\widetilde{S}}(-\widetilde{H}))$, whence
Lemma \ref{lemma:restanti} yields
\[ h^1(\T_S(-H)) = \begin{cases} 2 & \; \mbox{if} \; (S,H) \in \E_{6,2} \\
                                4 & \; \mbox{if} \; (S,H) \in \E_{4,2} \\
                                6  & \; \mbox{if} \; (S,H) \in \E_{3,2}, \end{cases}\] 
which is the dimension of a general fiber of $\chi_{g,2}$ by 
Corollary \ref{cor:geninj}. Comparing dimensions of $\P_{g,2}$ and $\R_g$  yields the rest.

 (ii) Recalling that  $H \sim 2E_1+E_{1,2}$, we first apply Lemma
\ref{lemma:trick4} with $G_1=E_1$ and $G_2=E_{1,2}$ to compute
$h^1(\T_{\widetilde{S}}(-\widetilde{H}))$.
We have $H-G_1-G_2=E_1$, whence $\gamma=0$. We have $2G_1+2G_2-H=E_{1,2}$, whence $\delta=2$. Finally, the multiplication map $\mu_{\widetilde{E}_1+ \widetilde{E}_{1,2},\widetilde{E}_1}$ is surjective by Lemma \ref{lemma:mumford2}, as
\[ h^1(\widetilde{E}_1+ \widetilde{E}_{1,2}-\widetilde{E}_1)=h^1(\widetilde{E}_{1,2})=h^1(E_{1,2})+h^1(E_{1,2}+K_S)=0.\]
Hence, $\epsilon=0$.  Thus, $h^1(\T_{\widetilde{S}}(-\widetilde{H}))=6$ by Lemma \ref{lemma:trick4} and the result follows as  in (i). 
\end{proof}

To finish the proof of Theorem \ref{mainthm2} we will  have to study the cases $\phi=2$ of odd  genus  $g \geq 5$   apart from $\chi^{(I)}_{5,2}$.  We will do this in  Sections
\ref{sec:fi2-II-n} and \ref{sec:fi2-II} after a technical result in the next section.  Theorem  \ref{mainthm2} will follow from Propositions  \ref{prop:fi2gpari}, \ref{prop:fi2g>5} and \ref{prop:fi2II}.

\section{A technical result} \label{sec:tec}

 We here  give a result that we will need
in the next section, where we will bound the fiber dimension of a moduli map by specializing  to a union $C \cup \Gamma$ of a smooth curve $C$ and a rational curve $\Gamma$ and using knowledge of the fiber dimension over  $C$. 

 Although  we will need the result in the  case  $X$  is an Enriques--Fano threefold, we formulate it in a more general setting. Its proof is independent of the rest of the paper and its reading can be postponed.

\begin{lemma} \label{lemma:lemmazzo}
  Let $X$ be a  normal  projective threefold and $\L$ a big and nef line bundle on $X$ such that the general member of $|\L|$ is a smooth, regular surface. Assume that there is a smooth surface $S_0 \in |\L|$ containing a smooth irreducible rational curve $\Gamma_0$ such that:
  \begin{itemize}
  \item[(i)] $\kod(S_0) \geq 0$  (where $\kod$ denotes the Kodaira dimension);
\item[(ii)] the general element in $|\L|$ does not contain any deformation of $\Gamma_0$;
\item[(iii)] $\L|_{S_0} \sim M+N$  such that $M$ and $N$ are effective and nontrivial and $M$ is globally generated;  moreover, $\Gamma_0 \cdot M >0$; 
\item[(iv)] there is  a smooth, irreducible  nonrational  
 $D \in |N|$    such that  $h^0(\O_D(\Gamma_0))=1$. 
  \end{itemize}
 Then, possibly up to substituting the pair $(S_0,\Gamma_0)$ with a deformation of it keeping $S_0$ inside $|\L \* \I_D|$ (which automatically maintains 
$N \sim D$ and $M \sim \L|_{S_0}-N$), the following holds:
For general $C \in |M|$,  the linear system  $|\L \* \I_{D \cup C}|$ is a pencil with base locus $D \cup C$ and either
\begin{itemize}
\item[(a)] $\Gamma_0$ does not deform to a general member of $|\L \* \I_{D \cup C}|$, or
\item[(b)] $\Gamma_0$ deforms to a general member of $|\L \* \I_{D \cup C}|$ in such a way that the intersection $\Gamma_t \cap C \neq \Gamma_0 \cap C$ for the general deformation $\Gamma_t$ of $\Gamma_0$.
\end{itemize}
\end{lemma}

\begin{proof}
  Let $\pi:\widetilde{X} \to X$ be  a resolution of singularities  of $X$ (which is an isomorphism on the smooth locus of $X$). 
 Arguing precisely as in the proof of Lemma \ref{lemma:basicEF}, one finds that $h^1(\O_{\widetilde{X}})=0$ and $H^0(\widetilde{X}, \pi^*\L) \cong H^0(X,\L)$.  We can therefore without loss of generality assume that $X$ is smooth and $h^1(\O_X)=0$. 

Let $S_0$, $\Gamma_0$, $M$, $N$ and  $D \in |N|$ be as in the statement. From
\[
\xymatrix{0 \ar[r] & \O_X \ar[r] &  \L \* \I_D \ar[r] &  \L|_{S_0}-D \cong M \ar[r] &  0,}
\] and the fact that $M$ is globally generated and $h^1(\O_X)=0$, we see that 
$|\L \* \I_D|$ is base point free off $D$ and 
$  \dim |\L \* \I_D| = \dim |M|+1 \geq 2$.
For any $C \in |M|$ we see from
\[
\xymatrix{0 \ar[r] & \O_X \ar[r] &  \L \* \I_{D \cup C} \ar[r] &  \L|_{S_0}-D-C \cong \O_{S_0} \ar[r] &  0,}
\]
that $|\L \* \I_{D \cup C}|$ is a pencil (containing the element $S_0$) with base locus $D \cup C$. Conversely, for any pencil $\Lambda \subset |\L \* \I_D|$ containing  $S_0$, the base locus is $C_{\Lambda} \cup D$ for some $C_{\Lambda} \in |M|$. Hence, giving a pencil $\Lambda \subset |\L \* \I_D|$ containing the element $S_0$ is equivalent to giving a curve $C \in |M|$ (which will be the base curve off $D$ of $\Lambda$). Note that $\Gamma_0 \not \sub C_{\Lambda} \cup D$ for general $\Lambda$, since $|M|$ is base point free and $D  \neq  \Gamma_0$ by assumption (iv).  
A general pencil $\Lambda \subset |\L \* \I_D|$ containing $S_0$ therefore does not have $\Gamma_0$ in its base locus. 

 Denote by $\R$ the  union of the components  of the incidence variety
 \[ \{ (\Gamma,S) \; | \; [\Gamma] \in \Hilb X,  S \in |\L \* \I_D|, \Gamma \subset S\} \subset \Hilb X \x |\L \* \I_D| \]
 containing $(\Gamma_0,S_0)$, where we denote by $[\Gamma]$ the point corresponding to the curve $\Gamma$ in the Hilbert scheme. We further denote   by $p: \R \to |\L \* \I_D|$ the natural projection, which is generically finite as a general member in $|\L \* \I_D|$ has nonnegative Kodaira dimension by assumption (i),  whence  it contains at most finitely many curves that are deformations of $\Gamma_0$. If $p$ is not dominant, we end up in case (a), taking the pencil generated by $S_0$ and a general element not in the image of $p$. We  may thus  henceforth assume $p$ is dominant. In particular, we can assume that $S_0$ is general in $|\L \* \I_D|$. Indeed, for general $S_t \in |\L \* \I_D|$, one may define $N_t:=\O_{S_t}(D)$ and $M_t:=\L|_{S_t}-D$, and properties (i)-(iv)  are preserved passing from $(S_0,\Gamma_0)$ to a general $(S_t,\Gamma_t)$.

Let now $\Lambda \subset |\L \* \I_D|$ be  a general pencil containing $S_0$.  
Its general member therefore does not contain $\Gamma_0$ and contains only finitely many rational curves deformations of $\Gamma_0$, as it has nonnegative Kodaira dimension by assumption (i). 
 For $\lambda \in \Lambda$ we denote by $S_{\lambda}$ the corresponding surface. 
Let $\R_{\Lambda}$ be any irreducible component of the incidence variety
\[ \{ (\Gamma,\lambda) \; | \; [\Gamma] \in \Hilb X,  \lambda  \in \Lambda, \Gamma \subset S_{\lambda}\} \subset \Hilb X \x \Lambda \]
containing $([\Gamma_0],0=S_0)$ and such that the second projection $p_{\Lambda}: \R_{\Lambda} \to \Lambda$ is dominant. Such a component exists since $p$ is dominant. Moreover, $p_{\Lambda}$ is generically finite by what we  said above.

Assume $p_{\Lambda}$ is not generically  injective.  Then the general $S \in \Lambda$ contains at least two curves that are deformations of $\Gamma_0$.  As we assume that $S_0$ is general in $|\L \* \I_D|$,  it is not a branch point of $p$, so that  the limit curves on $S_0$ are all distinct.  As we assume that  $\Lambda$ is general,  the curve $C_{\Lambda}$ (the base curve of $\Lambda$ off $D$) is general in $|M|$ and therefore  does not pass through the intersection points of the finitely many curves  on $S_0$  in the component of $\Hilb X$ containing $[\Gamma_0]$.  This  forces the intersection points $\Gamma \cap C_{\Lambda}$ to vary as $(\Gamma,S)$ varies in $\R_{\Lambda}$;  indeed, $\Gamma$ cannot specialize to a curve containing $C_{\Lambda}$, because the latter is not rational for general $\Lambda$, as it moves on $S_0$ and $\kod(S_0) \geq 0$.  We thus end up in case (b).

Assume therefore that $p_{\Lambda}$ is generically  injective.   In particular, 
there is a dense, open subset $\Lambda^{\circ} \subset \Lambda$ such that for all $\lambda \in \Lambda^{\circ}$, the surface $S_{\lambda}$ is smooth and contains a   distinguished  curve $\Gamma_{\lambda} \neq \Gamma_0$ that is a deformation of $\Gamma_0$. 
Consider the irreducible closed surface
$R_{\Lambda}:= \overline{\cup_{\lambda \in \Lambda^{\circ}} \Gamma_{\lambda}}  \subset X$.
This surface can also be described as the image in $X$ of the universal family over the image of $\Lambda \to \Hilb X$. 
Let us study the intersection $R_{\Lambda} \cap S_{\lambda}$ for general $\lambda \in \Lambda$.  Clearly, $R_{\Lambda} \cap S_{\lambda}$ is a curve containing  $ \Gamma_{\lambda}$. 

If $R_{\Lambda} \cap S_{\lambda} = \Gamma_{\lambda}$ (set-theoretically), then the intersection is transversal for general $\lambda \in \Lambda$, as $\Lambda$ is base point free off $D \cup C_{\Lambda}$. Hence $\Gamma_{\lambda}=R_{\Lambda} \cdot S_{\lambda} = 
R_{\Lambda} \cdot \L$, and it would follow that a general member of $|\L|$ contains a deformation of $\Gamma_0$, contradicting (ii). 

Therefore, $R_{\Lambda} \cap S_{\lambda}$  contains  a curve $F_{\lambda}$ in addition to $\Gamma_{\lambda}$.  We claim that $F_{\lambda}$ does not vary with $\lambda$, and therefore  $F_{\lambda}=D$, $C_{\Lambda}$ or $D \cup C_{\Lambda}$.  Indeed, if $F_{\lambda}$ varies, it cannot lie in  
$\overline{R_{\Lambda} \setminus \cup_{\lambda \in \Lambda^{\circ}} \Gamma_{\lambda}}$, as it  consists of finitely many curves.  But then  $F_{\lambda}$, for general $\lambda$, must intersect $ \cup_{\lambda \in \Lambda^{\circ}} \Gamma_{\lambda}$ in infinitely many points, 
and must therefore lie in the base locus $D \cup C_{\Lambda}$ of $\Lambda$,  a contradiction.  Thus, $R_{\Lambda} \cap S_{\lambda}=F_{\lambda} \cup \Gamma_{\lambda}$, with $F_{\lambda}=D$, $C_{\Lambda}$ or $D \cup C_{\Lambda}$.
Moreover,  $R_{\Lambda} \cap S_{\lambda}$ contains $D$ (respectively, $C_{\Lambda}$) for general $\lambda \in \Lambda^{\circ}$, only if  $\Gamma_{\lambda} \cap D$ (resp., $\Gamma_{\lambda} \cap C_{\Lambda}$) varies with $\lambda$: indeed, 
the finitely many curves in
$\overline{R_{\Lambda} \setminus \cup_{\lambda \in \Lambda^{\circ}}
\Gamma_{\lambda}}$ are  rational, being  components of limit curves of the
$\Gamma_{\lambda}$ with $\lambda \in \Lambda^{\circ}$,
whereas $D$ is not rational by assumption (iv)  and 
neither is
$C_{\Lambda}$  as it moves on $S_0$. 
Thus $R_{\Lambda} \cap S_{\lambda}$  cannot contain $D$, as
  $\{\Gamma_{\lambda} \cap D\}_{\lambda \in \Lambda^{\circ}}$  would then
form  a nonconstant family of rationally equivalent cycles on $D$, whence
$h^0(\O_D(\Gamma_0)) \geq 2$, contradicting assumption (iv). 
 Hence $R_{\Lambda} \cap S_{\lambda}$  contains $C_{\Lambda}$,  and 
we end up in case (b).
\end{proof}

\section{The moduli maps  on $\EC_{g,2}^{(I)}$  for $g \geq 7$} \label{sec:fi2-II-n}

The main result of this section is the following, which  proves part  of Theorem \ref{mainthm2}.

\begin{proposition} \label{prop:fi2g>5}
The map 
$\chi_{g,2}^{(I)}$ is generically finite  if  $g \geq 9$ and has generically one-dimensional  fibers  if $g=7$.    
\end{proposition}

Recall that the irreducible component $\E_{g,2}^{(I)}$ occurs for all
odd $g$, and corresponds to polarizations  $\frac{g-1}{2}E_1+E_{1,2}$.
The proposition will be proved by semicontinuity, specializing the curves in $\EC_{g,2}^{(I)}$ to a union of a curve in $\EC_{g-1,2}$ and a smooth rational curve. We will therefore first develop  some auxiliary
results on $\E_{2k+2,2}$.
Recall that by \cite{cdgk}  these spaces  are irreducible
and that $H \sim kE_1+E_2+E_3$ for $(S,H) \in \E_{2k+2,2}$.

Let $k \geq 1$. We define  $\E_{2k+2,2}^{\circ} \subset \E_{2k+2,2}$  to be the subset   parametrizing pairs
$(S,H=kE_1+E_2+E_3)$ such that $E_1,E_2,E_3$ are nef and both
$|E_1+E_2+E_3|$ and $|E_1+E_2+E_3+K_S|$ map $S$ birationally onto a
sextic.  Nonemptiness of this locus follows from \cite[\S 7]{cos1}.

 We set $\P^\circ_{2k+2,2}:=p_{2k+2,2}^{-1}
(\E_{2k+2,2}^\circ) \subset \P_{2k+2,2}$ and denote by 
$c^\circ_{2k+2,2}: \P^\circ_{2k+2,2} \to \M_g$ the restriction of $c_{2k+2,2}$  to $\P^\circ_{2k+2,2}$. 
A key result is the following  stronger version of Proposition \ref{prop:fi2gpari}(i):

\begin{proposition} \label{prop:desvan}
   For any $C \in \im c^\circ_{2k+2,2}$, we have 
$\dim \bigl({c^\circ_{2k+2,2}}^{-1}(C) \bigr)=0$ if $k \geq 3$, whereas
${c^\circ_{6,2}}^{-1}(C)$ is equidimensional of dimension two. 
\end{proposition}

 To prove this, we first need an auxiliary result:

\begin{lemma} \label{lemma:nonspec}
  Let $(S,H \sim kE_1+E_2+E_3) \in \E_{2k+2,2}^{\circ}$.  Then, for  $\{\alpha,\beta,\gamma\}=\{1,2,3\}$, and any $l \in \ZZ$,  we have $h^i(lE_{\alpha}+E_{\beta}-E_{\gamma})=0$, $i=0,1,2$.
\end{lemma}

\begin{proof}
   Since $(lE_{\alpha}+E_{\beta}-E_{\gamma})^2=-2$, the statement is by Riemann-Roch and Serre duality equivalent to the fact that  the divisor $lE_{\alpha}+E_{\beta}-E_{\gamma}$ is not numerically equivalent to an effective divisor  for 
 any $l \in \ZZ$ and  $\{\alpha,\beta,\gamma\}=\{1,2,3\}$.

  By \cite[Def.~5.3.1 and (5.3.2)]{cos2} (see also 
\cite[\S 7]{cos1}), neither $E_{\alpha}+E_{\beta}-E_{\gamma}$ nor $E_{\alpha}+E_{\beta}-E_{\gamma}+K_S$ is linearly equivalent to an effective  divisor.
It  is clear that $lE_{\alpha}+E_{\beta}-E_{\gamma}$ cannot be numerically equivalent to an effective divisor if $l \leq 0$.

Assume therefore, to get a contradiction, that $lE_{\alpha}+E_{\beta}-E_{\gamma}$ is numerically equivalent to an effective divisor $\Delta$ for some $l \geq 2$. We claim that
\begin{equation}
  \label{eq:conte}
  2E_{\alpha}-\Delta >0.
\end{equation}
This yields the desired contradiction, as $E_{\gamma} \cdot (2E_{\alpha}-\Delta)
=1-l <0$.

Let us prove \eqref{eq:conte} by induction on $l$. We may assume that
$\Delta$ does not contain any multiple of $E_{\alpha}$ or $E_{\alpha}+K_S$, as
otherwise $(l-1)E_{\alpha}+E_{\beta}-E_{\gamma}$ would be numerically equivalent to an
effective divisor.
Since $\Delta \cdot E_{\alpha}=0$, 
we have
\begin{equation}
\label{eq:sotto}
{\Delta'}^2 \leq -2 
\quad \text{for every effective subdivisor }
\Delta' \text{ of } \Delta
\end{equation}
by \eqref{kl-HIT}.
Pick a $(-2)$-curve $ R  \leq \Delta$. Since $|2E_{\alpha}|$ is an
elliptic pencil and $ R  \cdot E_{\alpha}=0$, it follows that $ R $
must be part of a fiber of the elliptic fibration defined by $|2E_{\alpha}|$,
whence $2E_{\alpha}- R  >0$. Set $\Delta':=\Delta- R $. If $\Delta'>0$,
then, using \eqref{eq:sotto}, we find
\[ -2=\Delta^2={\Delta'}^2+ R ^2+2\Delta' \cdot  R  \leq
-4+2\Delta' \cdot  R ,\]
whence $\Delta' \cdot  R  \geq 1$. Hence, there exists a 
$(-2)$-curve $ R ' \leq \Delta'$ such that $ R ' \cdot  R 
\geq 1$; more precisely, we have $ R ' \cdot  R  = 1$, since
otherwise $( R + R ')^2 \geq 0$, contradicting
\eqref{eq:sotto}. Since $ R ' \cdot (2E_{\alpha}- R )=-1$, we must have
$2E_{\alpha}- R - R '>0$.  Repeating the procedure, if necessary,
eventually yields \eqref{eq:conte}. 
\end{proof}

\begin{proof}[Proof of Proposition \ref{prop:desvan}]
   By Lemmas  \ref{lemma:diff} and \ref{lemma:posfibfi2}(v), the result will follow if we prove that 
 for any $(S,H) \in \E_{2k+2,2}^\circ$ and $k\geq 2$, we have
\begin{equation} \label{eq:desvan}
  h^1(\T_S(-H))+h^1(\T_S(-(H+K_S))=
\begin{cases} 0, & \;   \mbox{ if}  \; \; \; k \geq 3, \\
4, & \;   \mbox{ if}  \; \; \;  k=2.
\end{cases}
\end{equation}
 (Indeed, when $k=2$, since $(S,H)$ and $(S,H+K_S)$ are both general in $\E^{\circ}_{6,2}$ we have  that $h^1(\T_S(-H))$ and $h^1(\T_S(-(H+K_S))$ are equal, and 
\eqref{eq:desvan} implies they are both equal to 2.)

  We apply Lemma \ref{lemma:tricknuovo} with $H=kE_1+E_2+E_3$, $F_1=E_1$ and 
$F_2=E_2$,  and  \eqref{eq:h1soprasotto}.  

As  $H-2F_1  \sim  (k-2)E_1+E_2+E_3$ is big and nef, we have $h^1(H-2F_1)=h^1(H-2F_1+K_S)=0$. We have  $H-2F_2  \sim  kE_1+E_3-E_2$,  so that  $h^1(H-2F_2)=0$ by Lemma \ref{lemma:nonspec} and $h^1(H-2F_2+K_S)=0$ by the same lemma applied with $E_3$ replaced by $E_3+K_S$. Hence  $\alpha=0$ by Lemma \ref{lemma:nonspec}. 

   As  $H-2F_1-2F_2  \sim  (k-2)E_1+E_3-E_2$,   Lemma \ref{lemma:nonspec}   (possibly applied again with $E_3$ replaced by $E_3+K_S$)   and  \eqref{eq:bbeta3} yield that  $\beta=h^0(4F_1+4F_2-H)+h^0(4F_1+4F_2-H+K_S)$. We  have
$4F_1+4F_2-H  \sim  (4-k)E_1+3E_2-E_3$.  As  $E_2 \cdot (4F_1+4F_2-H)=3-k$, we see that $\beta=0$  if  $k \geq 4$. Moreover, $\beta=0$  if   $k=3$ by Lemma \ref{lemma:nonspec}.  If   $k=2$,   we claim that $\beta=4$.
Indeed, as $(4F_1+4F_2-H)^2=2$,  the claim follows if we prove that 
$h^1(D)=0$ for $D \equiv 2E_1+3E_2-E_3$. If, by contradiction, $h^1(D) >0$, there is by \cite{klvan} an effective divisor
$\Delta$ such that $\Delta^2=-2$ and $\Delta \cdot D \leq -2$. Then $(D-\Delta)^2 \geq 4$ and $(D-\Delta) \cdot (E_1+2E_2) \leq D \cdot (E_1+2E_2)=4$. Since
$(E_1+2E_2)^2=4$, the Hodge index theorem yields $D-\Delta \equiv E_1+2E_2$, whence $\Delta \equiv E_1+E_2-E_3$, contradicting Lemma \ref{lemma:nonspec}. 
We have therefore proved that $\beta=4$ when $k=2$.

By Lemma \ref{lemma:tricknuovo}, we have 
$h^1(\T_{\widetilde{S}}(-\widetilde{H}))= 0$ if $k \geq 3$ and
$h^1(\T_{\widetilde{S}}(-\widetilde{H}))= 4$ if $k=2$, 
and  \eqref{eq:desvan} 
follows from \eqref{eq:h1soprasotto}.
\end{proof}

 The next key ingredient in the proof of Proposition \ref{prop:fi2g>5} is the identification of a suitable sublocus of {\it nodal} Enriques surfaces.

\begin{proposition} \label{prop:nuovoesiste}
 The closed subset
 $\E_{2k+2,2}' \subset \E_{2k+2,2}^\circ$ 
parametrizing  $(S,H)$ such that $S$ contains a smooth rational
curve $\Gamma$ with $\Gamma \cdot E_1=0$ and $\Gamma \cdot
E_2=\Gamma \cdot E_3=1$  (possibly after rearranging indices when $k=1$) is irreducible of codimension one.
 Moreover, for general $(S,H)$ in $\E_{2k+2,2}'$, we have $\Gamma \cap E_2 \cap E_3  =\emptyset$. 
\end{proposition}

Again, to prove this we need an auxiliary result:

\begin{lemma} \label{lemma:esiste}
  There exists an Enriques surface $S$ containing three nef, primitive isotropic divisors $E_1$, $E_2$ and $E_3$ and  smooth rational  curves $\Gamma$ and $\Gamma'$  such that
  \begin{compactitem}
  \item[(i)] $2E_1 \sim \Gamma +\Gamma'$, with $\Gamma \cdot \Gamma'=2$, and  the latter intersection is transversal;  
\item[(ii)] $E_2 \cdot E_3=1$, $\Gamma \cdot E_2=\Gamma' \cdot E_2=\Gamma \cdot E_3=\Gamma' \cdot E_3=1$ (whence $E_1 \cdot E_2=E_1 \cdot E_3=1$);
\item[(iii)]  the elliptic pencils $|2E_2|$ and $|2E_3|$ have no reducible fibers.   
\item[(iv)]   $|E_1+E_2+E_3|$  is ample and  maps $S$ birationally onto a sextic surface. 
 \end{compactitem}
\end{lemma}

\begin{proof}
   By \cite[Lemma 3.2.1]{cos3} there exists an Enriques surface $S$ with ten elliptic pencils $|2F_i|$ and ten smooth rational curves $D_i$, with $1\leqslant i\leqslant 10$, such that 
  \[ D_i \cdot F_j =1 \; \mbox{for} \;  i \neq j; \; \; D_i \cdot F_i =3; \; \;
   F_i \cdot F_j =1 \; \mbox{for} \; i \neq j; \; \;
 D_i \cdot D_j =2 \;  \mbox{for} \; i \neq j.\]
Moreover, by \cite[Rem. p.~747]{cos3}, the elliptic pencils $|2F_i|$ have no reducible fibers, and by \cite[Prop. 3.2.6]{cos3} the complete linear system $|D_i+D_j+D_k|$, for any distinct $i,j,k$, defines a degree two  
morphism  of $S$ onto a Cayley cubic surface in $\PP^3$. Thus, by \cite[Thm. 7.2 and (7.7.1)]{cos1}, the surface $S$ can equivalently be realized as the minimal desingularization of the double cover of $\PP^2$ branched along a {\it Wirtinger sextic} (a sextic with six double points at the vertices of a complete quadrilateral) and the edges of its complete quadrilateral. The curves $D_i, D_j, D_k$ are the inverse images of the three diagonals of the quadrilateral, whence they intersect pairwise transversely in two points. 

By \cite[Lemma 1.6.2]{cos2} there exists a $B \in \Pic (S)$ such that
$3B \sim F_1+\cdots+ F_{10}$. Set $F_{ij}:=B-F_i-F_j$, for $i \neq
j$. Then $F_{ij}^2=0$.  We get $D_i \cdot B=4$ for all $i$, whence
$D_i \cdot F_{ij}=0$ for $i \neq j$. As $(D_i+D_j)^2=(D_i+D_j)\cdot
F_{ij}=0$, we must have $(D_i+D_j) \equiv q F_{ij}$ for some $q \in
\QQ$ by \eqref{kl-HIT}, 
and dotting both sides with $F_i$ yields $q=2$. 
In particular, $F_{ij}$ is nef, so that $|2F_{i,j}|$ is an elliptic pencil. Hence, one necessarily has $2F_{i,j} \sim D_i+D_j$. 
The divisors $E_1:=F_{45}$, $E_2:=F_2$, $E_3:=F_3$, 
$\Gamma:=D_4$ and $\Gamma':=D_5$   satisfy  properties (i)--(iii).  Moreover, (iii) implies that both $E_2$ and $E_3$ has positive intersection with any $(-2)$--curve, whence $E_1+E_2+E_3$ is ample. By
\cite[\S 7]{cos1} or \cite[(5.3.2)]{cos2}, either $|E_1+E_2+E_3|$ or
$|E_1+E_2+E_3+K_S|$ maps $S$ birationally onto a sextic surface,
whence (iv) follows possibly by replacing any of $E_i$ with $E_i+K_S$. 
\end{proof}

\begin{proof}[Proof of Proposition \ref{prop:nuovoesiste}]
   We  argue as in \cite[ \S 5]{cdgk}. 
Fix homogeneous coordinates $(x_0:x_1:x_2:x_3)$ on $\PP^3$ and let
 $T =Z(x_0x_1x_2x_3)$  
 be the  \emph{coordinate tetrahedron}.  We label by
$\ell_1,\ell_2,\ell_3,\ell'_1,\ell'_2,\ell'_3$ the edges of $T$, in
such a way that $\ell_1,\ell_2,\ell_3$ are coplanar, and $\ell'_i$ is
skew to $\ell_i$ for all $i=1,2,3$.

Consider the linear system $\mathcal S$ of surfaces of degree  $6$  singular along the edges of $T$  (called 
\emph {Enriques sextics}). They  have equations of the form
\begin{equation*} \label{eq:sigma}
c_3(x_0x_1x_2)^2+c_2(x_0x_1x_3)^2+c_1(x_0x_2x_3)^2+c_0(x_1x_2x_3)^2+Qx_0x_1x_2x_3=0,
\end{equation*} 
where $Q=\sum_{i \leqslant j}q_{ij}x_ix_j$. 
This shows that $\dim(\mathcal S)=13$ and we may identify $\mathcal S$ with the $\PP^ {13}$ with  homogeneous coordinates
\[ q=(c_0:c_1:c_2:c_3:q_{00}:q_{01}:q_{02}:q_{03}:q_{11}:q_{12}:q_{13}:q_{22}:q_{23}:q_{33}). \]
  As in  \cite[ \S 5]{cdgk} we have a dominant rational map
  $\sigma_{2k+2,2}: \mathcal S \dashrightarrow \E_{2k+2,2}$,  which assigns to a general $\Sigma \in \mathcal S$ the pair $(S,H)$, where $\varphi:S\to \Sigma$ is the normalization and  $H=k\varphi^*(\ell_1)+\varphi^*(\ell_2)+\varphi^*(\ell_3)$.  Indeed, any $(S,H=kE_1+E_2+E_3)\in \E_{2k+2,2}$ such that $|E_1+E_2+E_3|$ is  ample and  birational lies in the image of  $\sigma_{2k+2,2}$, because the  image $\Sigma$ of $S$ via the map $\varphi:=\varphi_{E_1+E_2+E_3}$ is singular  precisely  along the edges of $T$,  cf. \cite[Thm. 4.6.3]{cd},  with $\ell_i=\varphi(E_i)$, after a suitable change of coordinates. 
  Also  note that the image of $\sigma_{2k+2,2}$ contains pairs $(S,H=kE_1+E_2+E_3$) satisfying the conditions of Lemma \ref{lemma:esiste}, because of property  (iv)  therein.

The fiber $\sigma_{2k+2,2} ^ {-1}(S,H)$ consists of the orbit of $\Sigma=\varphi(S)$ via the $3$--dimensional group of projective transformations  fixing $T$.

We denote by $\F$ the family of smooth conics  $F \subset \PP^3$ such that 
 $F$ does not contain the vertex $\ell'_1\cap\ell'_2\cap\ell'_3$ of $T$ and such that  $F$ meets the edges $\ell_2$, $\ell'_2$, $\ell_3$ and $\ell'_3$ exactly once and does not meet  $\ell_1$ and $\ell'_1$.

\begin{claim} \label{cl:nw01}
(a) The variety $\F$ is irreducible and $4$-dimensional.

(b) Each $F \in \F$ is contained in an $8$-dimensional linear system of Enriques sextics.
\end{claim}

\begin{proof}[Proof of the claim]
(a) Each $F$ in $\F$ spans a plane  intersecting   the edges  of $T$ in six points. 
The set of plane conics through  four of these six points is a $\PP^1$,  proving (a).

(b) The linear system $\mathcal S$ of the Enriques sextics cuts out on each $F \in \F$ a linear system of divisors with base locus (containing) $T \cap F$ and a moving part of degree (at most) $4$, whence of dimension at most $4$. Hence $F$ is contained in a linear system $\mathcal{S}_F$ of Enriques sextics  of dimension at least $8$. 

We claim that  for each $F \in \F$, one has $\dim (\mathcal S_F)=8$. 
Consider the restriction rational map $\rho_F: \mathcal{S} \dasharrow {\mathcal{S}}|_F$, whose indeterminacy locus is $\mathcal S_F$. Pick any Enriques sextic $\Sigma$ containing $F$ and let $S$ be its normalization. We consider by abuse of notation $F$ as a curve in $S$. Then $\rho_F$ factors through the restriction $\rho_S$ to $S$ and the restriction $\rho_{S,F}$ from $S$ to $F$, i.e.,
 $\xymatrix{\rho_F: \mathcal{S} \stackrel {\rho_S}\dasharrow   \mathcal{S}|_S  \stackrel {\rho_{S,F}} \dasharrow  \mathcal{S}|_F}. 
$ 
The indeterminacy locus of $\rho_S$ is just the point $[\Sigma]$. Therefore,  denoting by  $\mathcal S_{S,F}$ the indeterminacy locus of $\rho_{S,F}$, we have $\dim (\mathcal S_F)=\dim(\mathcal S_{S,F})+1$. So we have to prove that 
$\dim(\mathcal S_{S,F})=7$.

The restricted linear system $\mathcal{S}|_S$ is 
$|2(E_1+E_2+E_3)|$; indeed, it is 
the sublinear system of $|6(E_1+E_2+E_3)|$ having base locus twice the
sum of the pullback of the edges of the tetrahedron, which is
\[2\bigl(E_1+E_2+E_3+(E_1+K_S)+(E_2+K_S)+(E_3+K_S)\bigr) \sim
4(E_1+E_2+E_3).\]
Hence $\mathcal S_{S,F}$ is the projectivization of the kernel of the
restriction map
\[ H^0(\O_S(2(E_1+E_2+E_3)) \longrightarrow H^0(\O_{F}(2(E_1+E_2+E_3)), \]
which is $|2(E_1+E_2+E_3)-F|$. Set $D:=2(E_1+E_2+E_3)-F$. Then
$D^2=14$. We want to prove that $\dim (\mathcal S_{S,F})=\dim(
|D|)=7$, which amounts to proving that $h^1(D)=0$. Assume
$h^1(D)>0$. Then, by \cite{klvan}, there exists an effective divisor
$\Delta$ such that $\Delta^2=-2$ and $\Delta \cdot D \leq -2$. In
particular, $\Delta \cdot F \geq 2$. Since $F$ is mapped by the
morphism $\varphi$ defined by $|E_1+E_2+E_3|$ to a smooth conic,
$\Delta$ cannot be contracted by $\varphi$. Hence, $\Delta \cdot
(E_1+E_2+E_3)>0$. It follows that $(D-\Delta) \cdot (E_1+E_2+E_3) < D
\cdot (E_1+E_2+E_3)=10$. But this contradicts the Hodge index theorem,
since $(D-\Delta)^2 (E_1+E_2+E_3)^2 \geq 16 \cdot 6=96$.
\end{proof}

 By  the claim   the incidence variety 
 $\G : = \{ (F,\Sigma) \in \F \x \mathcal{S} \; | \; F \subset \Sigma\}$ 
is irreducible of dimension $12$.  Denote by $\E''_{2k+2,2}$ the image of the projection $\G \to \mathcal{S}$ followed by $\sigma_{2k+2,2}: \mathcal{S} \dashrightarrow \E_{2k+2,2}$, which is nonempty by Lemma \ref{lemma:esiste}, as already remarked.  The projection has finite fibers, since an Enriques surface contains only finitely many conics with respect to a given polarization, and $\sigma_{2k+2,2}$ has three-dimensional fibers,  whence $\E''_{2k+2,2}$ is irreducible of dimension nine. 
It parametrizes by construction all pairs $(S,H=kE_1+E_2+E_3)$ such that
$E_1,E_2,E_3$ are nef, $|E_1+E_2+E_3|$ is  ample and  birational and $S$ contains 
a smooth rational
curve $\Gamma$ with $\Gamma \cdot E_1=0$ and $\Gamma \cdot
E_2=\Gamma \cdot E_3=1$  (possibly after rearranging indices when $k=1$). Since it is irreducible, its general element has the property that also
$|E_1+E_2+E_3+K_S|$ is birational. Thus, $\E'_{2k+2,2}=\E''_{2k+2,2} \cap \E_{2k+2,2}^{\circ}$ is nonempty, whence irreducible of dimension nine, as stated.   The last assertion of the proposition follows as $F \cap \ell_2 \cap \ell_3 =\emptyset$ for general $F \in \F$. 
\end{proof}

We set 
$\P'_{2k+2,2}:=p_{2k+2,2}^{-1} (\E_{2k+2,2}') \subset \P^{\circ}_{2k+2,2}$,  which is  irreducible of codimension one in $\P^{\circ}_{2k+2,2}$.

\begin{proof}[Proof of Proposition \ref{prop:fi2g>5}]

Let  $(S,kE_1+E_2+E_3) \in \E_{2k+2,2}'$ be general, $k \geq 2$. 
 Set $H:=kE_1+E_2+E_3+\Gamma$. Then $H$ is big and nef, but not ample, as $\Gamma \cdot H=0$. 

 Consider  $\overline{\E}_{2k+3,2}^{(I)}$, the closure of
 $\E_{2k+3,2}^{(I)}$ in the moduli space of pairs $(X,L)$ where $X$ is a smooth Enriques surface and $L$ is a {\it big and nef} line bundle on $X$.  (The existence of such a moduli space is indicated for instance in \cite[\S 5.1.4]{Huy} for $K3$ surfaces and the case of Enriques surfaces is analogous; see also \cite{dol}.)  We claim that $(S,H)$
 lies in $\overline{\E}_{2k+3,2}^{(I)}$. Indeed, set $B:= E_2+E_3+\Gamma$. Then $B$ is nef with $B^2=4$  and $\phi(B)=2$ (as  
$E_2 \cdot B=E_3 \cdot B= 2$).  Since also $E_1 \cdot B=2$, we may write $B\sim E_1+E_{1,2}$ for some  effective  isotropic primitive $E_{1,2}$ satisfying
$E_1 \cdot E_{1,2}=2$. Thus
 $ H \sim  kE_1+B \sim (k+1)E_1+E_{1,2}$, proving the claim.

Denote by $\overline{\P}_{2k+3,2}^{(I)}$ the partial compactification of
$\P_{2k+3,2}^{(I)}$ parametrizing  triples 
$(S,H,C)$, where $(S,H)$ lies in $\overline{\E}_{2k+3,2}^{(I)}$  and $C \in |H|$ has at 
most nodes as
singularities.  Denote by 
$\overline{c}_{2k+3,2}^{(I)}: \overline{\P}_{2k+3,2}^{(I)} \to
\overline{\M}_{2k+3}$ the extension of $c_{2k+3,2}^{(I)}$.
Pick a general $(S,\O_S(C),C) \in \P'_{2k+2,2}$ and consider $C':= C
\cup \Gamma$. Then
$(S,\O_S(C'),C') \in \overline{\P}_{2k+3,2}^{(I)}$. By  Proposition
\ref{prop:desvan}, the fiber $c_{2k+2,2}^{-1}(C)$ is finite for $k \geq
3$. Since $\Gamma$ does not move on any Enriques surface, also the
fiber $(\overline{c}_{2k+3,2}^{(I)})^{-1}(C')$ is finite. Hence,
$c_{2k+3,2}^{(I)}$ is generically finite for $k \geq 3$, that is, $g
\geq 9$, and so is $\chi_{g,2}^{(I)}$.

  Assume now $k=2$, that is, $g=7$. Then $(S,2(E_1+E_2+E_3))$ is extendable to the classical Enriques--Fano threefold $(Y,\O_Y(1))$ in $\PP^{13}$ 
by Lemma \ref{lemma:tutteclassic}. Let $D \in |E_2+E_3|$ be general. Then 
Lemma \ref{lemma:proj} implies that $(S,2E_1+E_2+E_3)$ is extendable to an Enriques--Fano threefold $(Y',\L)$ and the members in $|\L|$ are in one-to-one-correspondence to the members in $|\O_Y(1) \* \I_D|$.  Since the hyperplane sections $S'$ of $Y$ such that
$\O_{S'}(1) \sim 2(E'_1+E'_2+E'_3)$ with $(S',2E'_1+E'_2+E'_3) \in \E'_{6,2}$ (possibly after rearranging the $E'_i$s) form a hypersurface in $|\O_Y(1)|$ by  Proposition \ref{prop:nuovoesiste}, the  members in $|\L|$ yielding  elements in 
$\E'_{6,2}$ form a subset $\N$ of codimension at most one in  $|\L|$.
 Hence, a general pencil in  $|\L|$  contains a general element in $\N$. By the proof of Proposition \ref{prop:pencil},   two  general members of the pencil are not isomorphic. This means that we  may find   a pencil $|\O_Y(1) \* \I_{C \cup D}|$ of hyperplane sections of $Y$ containing $S$ such that
$(S,\O_S(C),C)$ is general in $\EC'_{6,2}$ and two general surfaces in the pencil are not isomorphic, that  is, we  have a finite  rational  map
$\mathfrak{a}:  |\O_Y(1) \* \I_{C \cup D}| \dashrightarrow  {c^{\circ}_{6,2}}^{-1}(C)$. We also have a  rational   map 
$\mathfrak{b}:  {\left(\overline{c}_{7,2}^{(I)}\right)}^{-1}(C')  \dashrightarrow  {c^{\circ}_{6,2}}^{-1}(C)$, forgetting $\Gamma$, which is finite, as $\Gamma$ does not move on any Enriques surface. Accepting for a moment that the hypotheses of Lemma \ref{lemma:lemmazzo} are satisfied (for $X=Y$, $\L=\O_Y(1)$,  $S_0=S$, $\Gamma_0=\Gamma$, $M \sim \O_S(2E_1+E_2+E_3)$ and $N \sim \O_S(E_2+E_3)$), we obtain that $\mathfrak{b}$   restricted to a neighborhood of $[(S,\O_S(C'),C')]$ is not dominant.  Indeed, either (a) of Lemma \ref{lemma:lemmazzo} holds, in which case there are elements in the image of $\mathfrak{a}$ not containing any deformation of $\Gamma$, whence not lying in the image of $\mathfrak{b}$. Else,
(b) of Lemma \ref{lemma:lemmazzo} holds, in which case $\Gamma$ deforms to a general surface in the pencil $|\O_Y(1) \* \I_{C \cup D}|$, but in such a way that the moduli of $C'=C \cup \Gamma$  vary,  thus again yielding an element in the image of $\mathfrak{a}$ outside the image of $\mathfrak{b}$.  
By Proposition
\ref{prop:desvan}, this implies that ${\left(\overline{c}_{7,2}^{(I)}\right)}^{-1}(C')$ has a component of dimension $\leq 1$. By semicontinuity  of the dimension of the fibers of a morphism (see
\cite [Lemme (13.1.1)]{EGA}),  a general
fiber of $c_{7,2}^{(I)}$ (and of $\chi_{7,2}^{(I)}$) has dimension $\leq 1$. 
Lemma \ref{lemma:posfibfi2}(v) implies that equality holds, as desired.

Finally, we check the hypotheses in Lemma \ref{lemma:lemmazzo}.  The general hyperplane section of $Y$ is unnodal by Lemma \ref{lemma:tutteclassic}, whence (i) and (ii) are satisfied.  
We have that $M$ is globally generated by \cite[Prop. 3.1.6 and Thm. 4.4.1]{cd},
since $M$ is nef with $\phi(M)=E_1 \cdot M=2$, and $|N|$ contains a smooth curve $D$ by \cite[Prop. 3.1.6]{cd} and \cite[Prop. 8.2]{cos1}, as $E_2$ and $E_3$ are nef.
The fact that $h^0(\O_D(\Gamma))=1$ can be verified, using semicontinuity, by specializing $D$ to $E_2 +E_3$, considering
\[
\xymatrix{ 0 \ar[r] & \O_{E_3}(\Gamma-E_2) \ar[r] & \O_D(\Gamma) \ar[r] &  \O_{E_2}(\Gamma) \ar[r] & 0}
\]
and using that $h^0(\O_{E_2}(\Gamma))=1$ and 
$h^0(\O_{E_3}(\Gamma-E_2))=0$ by   the last assertion in  Proposition \ref{prop:nuovoesiste}.  
\end{proof}

\begin{corollary} \label{cor:image}
  The  maps  $\chi_{8,2} $,  $\chi_{7,2}^{(I)}$,  $\chi_{6,2}$ and  $\chi_{5,2}^{(I)}$   dominate $\R^0_g$. 
\end{corollary}

\begin{proof}
    The result follows from Lemma \ref{lemma:posfibfi2}(ii),(v),  Propositions \ref{prop:fi2gpari} and  \ref{prop:fi2g>5}.
\end{proof}

 Arguing similarly as above, we prove a result that we will need in the next section.

\begin{lemma} \label{lemma:mod}
  The map
  $c_{9,2}^{(II)^+}$ has some fibers of dimension $\geq 2$ whose general element $(S,H=4E_1+2E_2,C)$ has the property that  $E_1$ and $E_2$ are nef.
\end{lemma}

\begin{proof}
To keep notation consistent with the rest of the section, we switch the roles of $E_1$ and $E_2$ and write $H=2E_1+4E_2$ for pairs $(S,H) \in \E_{9,2}^{(II)^+}$.
We define a dense, open subset $\E_{9,2}^{\circ} \subset \E_{9,2}^{(II)^+}$
parametrizing pairs
$(S,H=2E_1+4E_2)$ such that $E_1$ and $E_2$ are nef.  In fact, $\E_{9,2}^{\circ}$ is non--empty because on the general Enriques surface $S$ there are smooth irreducible elliptic curves $E_1$, $E_2$, which are therefore nef, with $E_1\cdot E_2=1$.
The openess of  $\E_{9,2}^{\circ}$ follows from the fact that  $E_1,E_2$ being nef on $S$ is an open condition in the moduli space of Enriques surfaces. 

We set $\P^\circ_{9,2}:=p_{9,2}^{-1}
(\E_{9,2}^\circ) \subset \P_{9,2}$ and
$c_{9,2}^{\circ}:= c_{9,2}^{(II)^+}|_{\P^\circ_{9,2}}$. To prove the lemma, we want to find a curve $C$ in $ \im c_{9,2}^{\circ}$ with 
$\dim(c_{9,2}^{\circ})^{-1}(C) \geq 2$.

\begin{claim}\label{cl:11}
 There is an irreducible codimension-one sublocus $\E_{9,2}' \subset \EC_{9,2}^{\circ}$ parametrizing pairs $(S,H= 2E_1+4E_2)$ such that $2E_1 \sim \Gamma+\Gamma'$, where $\Gamma$ and $\Gamma'$ are smooth rational curves intersecting transversely in two points and such that $\Gamma \cdot E_2 = \Gamma' \cdot E_2=1$. 
\end{claim}

\begin{proof}[Proof of the Claim]
We argue as in the proof of Proposition \ref{prop:nuovoesiste}, from where we keep the notation. Consider   the  map  $\sigma_{9,2}^{(II)^+}: \mathcal S \dashrightarrow \E_{9,2}^{(II)^+}$ 
associating to a general $\Sigma \in \mathcal S$ the pair $(S,H)$, where $\varphi:S\to \Sigma$ is the normalization and  $H=2E_1+4E_2$, with
$E_i:=\varphi^*(\ell_i)$, $i=1,2$. Let $\E''_{9,2}$ denote the image of the projection of the incidence variety $\G$ to $\mathcal{S}$ followed by $\sigma_{9,2}^{(II)^+}$, which has (as before) dimension nine. It parametrizes 
pairs $(S,H=2E_1+4E_2) \in \EC_{9,2}^{\circ}$ such that $S$ contains a smooth rational curve $\Gamma$ satisfying $\Gamma \cdot E_1=0$ and $\Gamma \cdot E_2=1$ (and a nef, isotropic $E_3$ such that $E_1 \cdot E_3=E_2 \cdot E_3=\Gamma \cdot E_3=1$ and $|E_1+E_2+E_3|$ is birational). Since $|2E_1|$ is a pencil, we must have $\Gamma':=2E_1-\Gamma>0$. By Lemma \ref{lemma:esiste} there are elements in $\E''_{9,2}$ for which $\Gamma'$ is a smooth irreducible rational curve intersecting $\Gamma$ transversely in two points.  (Also note property (iii) in
Lemma \ref{lemma:esiste} implies that $E_2$ has positive intersection with any $(-2)$--curve, so that $2E_1+4E_2$ is ample.) 
  We let $\E_{9,2}'$ be the (open dense) locus of such pairs.  
\end{proof}

We have $ \dim(c_{9,2}^{\circ})^{-1}(C)
\geq 1$ for any  $C \in \im c_{9,2}^{\circ}$, as $\chi_{9,2}^{(II)^+}$ is not generically finite by Corollary \ref{cor:EF1}. 
Assume, by contradiction, that 
\begin{equation}
  \label{eq:contra}
  \dim(c_{9,2}^{\circ})^{-1}(C)=1 \; \; \mbox{for all} \; \; C \in \im c_{9,2}^{\circ}.
\end{equation}
We now argue as in the last part of the proof of Proposition \ref{prop:fi2g>5}
(for $k=2$). Denote by $\overline{\P}_{17,4}^{(II)^+}$ the partial compactification of
$\P_{17,4}^{(II)^+}$ parametrizing curves with at most nodes as
singularities and denote by 
$\overline{c}_{17,4}^{(II)^+}: \overline{\P}_{17,4}^{(II)^+} \to
\overline{\M}_{17}$ the extension of $c_{17,4}^{(II)^+}$.
Pick a general $(S,\O_S(C)=2E_1+4E_2,C) \in \P'_{9,2}$ and consider $C':= C
\cup \Gamma \cup \Gamma' \in |4E_1+4E_2|$. Then
$(S,\O_S(C'),C') \in \overline{\P}_{17,4}^{(II)^+}$. 

 Extend $(S,4(E_1+E_2))$ to Prokhorov's Enriques--Fano threefold $(W,\L)$  by  Proposition  \ref{lemma:prok4div-nuovo}. As in the last part of the proof of Proposition \ref{prop:fi2g>5}, we obtain for general $D \in |2E_1|$, a pencil $|\L  \* \I_{C \cup D}|$  in  $W$ containing $S$ such that
$(S,\O_S(C),C)$ is general in $\EC'_{9,2}$ and two general surfaces in the pencil are not isomorphic, that  is, we have a finite  rational  map
$ \mathfrak{a}:  |\L  \* \I_{C \cup D}| \dashrightarrow {c^{\circ}_{9,2}}^{-1}(C)$. We also have a finite  rational map $ \mathfrak{b}:  {\left(\overline{c}_{17,4}^{(II)^+}\right)}^{-1}(C')  \dashrightarrow  {c^{\circ}_{9,2}}^{-1}(C)$, forgetting $\Gamma \cup \Gamma'$. By Lemma \ref{lemma:lemmazzo} (with $X=W$,  $S_0=S$,  $\Gamma_0=\Gamma$ or $\Gamma'$, $M \sim \O_S(2E_1+4E_2)$ and $N \sim \O_S(2E_1)$) we obtain,  arguing as in the last part of the proof of Proposition \ref{prop:fi2g>5},
that  $\mathfrak{b}$  restricted to a neighborhood of $(S,\O_S(C'),C')$  is not dominant. By \eqref{eq:contra} this implies that $({\overline{c}_{17,4}^{(II)^+}})^{-1}(C')$ has a zero-dimensional component. By semicontinuity of the dimension of the fibers of a morphism (see
\cite [Lemme (13.1.1)]{EGA}), the general fiber of $c_{17,4}^{(II)^+}$ is zero-dimensional. Hence,  also $\chi_{17,4}^{(II)^+}$ is generically finite, contradicting Corollary
\ref{cor:EF1}. 
\end{proof}

\section{The moduli maps  on  $\EC_{g,1}$,   $\EC_{g,2}^{(II)}$, $\EC_{g,2}^{(II)^+}$ and $\EC_{g,2}^{(II)^-}$} \label{sec:fi2-II}

  The aim  of this section is to prove Theorem \ref{mainthm3} together with the following result,  which concludes the proof of Theorem \ref{mainthm2}. 

\begin{proposition} \label{prop:fi2II}
  (i) A general fiber of $\chi_{5,2}^{(II)^+}$ has dimension six; in particular 
 $\chi_{5,2}^{(II)^+}$ dominates $\R^{0,\mathrm{nb}}_5$. 
  
(ii) A general fiber of $\chi_{5,2}^{(II)^-}$ is four-dimensional;  in particular $\chi_{5,2}^{(II)^-}$ dominates $\D^0_5$.   

(iii) A general fiber of  $\chi_{7,2}^{(II)}$ has dimension $3$. 

(iv) A general fiber of $\chi_{9,2}^{(II)^+}$ has dimension $2$.

(v)  A general fiber of $\chi_{9,2}^{(II)^-}$ has dimension $1$. 

(vi) The moduli maps $\chi_{g,2}$ are generically finite on all irreducible components of $\P_{g,2}$ for all odd $g \geq 11$. 
\end{proposition}

 To prove the mentioned results, recall that  for $(S,H)$ in $\E_{g,2}^{(II)}$,
$\E_{g,2}^{(II)^+}$ or $\E_{g,2}^{(II)^-}$  (note that $\E_{g,2}^{(II)}$ occurs for $g \equiv 3 \; \mbox{mod} \; 4$, and 
$\E_{g,2}^{(II)^+}$ and $\E_{g,2}^{(II)^-}$ occur
 for $g \equiv 1 \; \mbox{mod} \; 4$) we have $H \equiv kE_1+2E_2$, $g=2k+1$, $k \geq 2$,  whereas for $(S,H)\in \E_{g,1}$, we have $H \sim (g-1)E_1+E_{2}$. Assume that $E_1$ and $E_2$ are nef and  consider the double cover $g:\widetilde{S} \to P:=\PP^1 \x \PP^1$ defined by 
$|\widetilde{E}_1+\widetilde{E}_2|$,  as in the beginning of \S \ref{sec:tools}. 

We denote any line bundle on  $P$  by the obvious notation $\O_{P}(a,b)$, its restriction to any effective divisor $D \subset P$ by $\O_D(a,b)$, and for any sheaf $\F$ on $P$, we set $\F(a,b):=\F \* \O_{P}(a,b)$. Recall that the branch divisor of $g$ is a smooth curve
$R \in |\O_{P}(4,4)|$.

\begin{lemma} \label{lemma:nuovo24.11}
For any $k, l \geq 1$ we have
$h^1(\T_{\widetilde{S}}(-k\widetilde{E}_1-l\widetilde{E}_2))
=h^1(\Omega_P(\log R)(k-2,l-2))$.
\end{lemma}

\begin{proof}
   By \cite[Lemma 3.16]{ev} we have
$g_* \T_{\widetilde{S}} \cong  \T_P(-2,-2) \+ \T_P\langle R \rangle$,
where $\T_P\langle R \rangle := \Omega_P (\log R)^{\vee}$  or is equivalently defined as  in \eqref{eq:deftlc}. 
We therefore have
 \begin{eqnarray*}
   h^1(\T_{\widetilde{S}}(-k\widetilde{E}_1-l\widetilde{E}_2)) & =  &  h^1(\T_{\widetilde{S}} \* g^*\O_P(-k,-l)) = h^1(\O_P(-k,-l) \* g_*\T_{\widetilde{S}}) \\
\hspace{-0.8cm} & \hspace{-1.2cm} = & \hspace{-0.8cm} h^1(\O_P(-k,-l) \* \T_{P}(-2,-2)) + h^1(\O_P(-k,-l) \* \T_P\langle R \rangle) \\ \hspace{-0.8cm} & \hspace{-1.2cm} = & \hspace{-0.8cm} h^1(\O_P(-k,-l-2)) + h^1(\O_P(-k-2,-l)) + h^1(\T_P\langle R \rangle(-k,-l)) \\
\hspace{-0.8cm} & \hspace{-1.2cm} = & \hspace{-0.8cm} h^1(\Omega_P(\log R)(k-2,l-2)).   
 \end{eqnarray*}
\end{proof}

The next lemma is the main ingredient in the proof of Proposition \ref{prop:fi2II}.

\begin{lemma} \label{lemma:dadimlemma1}
For any $(S,H)$ such that $H \equiv kE_1+2E_2$ with $E_1$ and $E_2$ nef and $E_1\cdot E_2=1$, we have
 \[
h^1(\T_S(-H))+h^1(\T_S(-(H+K_S)))=\begin{cases} 10, & \,\,  \mbox{if} \quad  k=2 \; (g=5) \\
6, & \,\,  \mbox{if}\quad k=3 \; (g=7) \\
                                                      3, & \,\,  \mbox{if} \quad  k=4 \; (g=9) \\
                                                      0, & \,\,  \mbox{if} \quad k \geq 5 \; (g \geq 11).
\end{cases}\]
\end{lemma}

\begin{proof}
We will compute $h^1(\T_{\widetilde{S}}(-\widetilde{H}))$ using Lemma
\ref{lemma:nuovo24.11} and use \eqref{eq:h1soprasotto}.
We  have   
\begin{equation} \label{eq:normdef2}
\xymatrix{ 0 \ar[r] & \Omega_P(k-2,0) \ar[r] 
& \Omega_P(\log R)(k-2,0) \ar[r] 
&  \O_{R}(k-2,0) \ar[r] &  0,}
\end{equation} 
 cf., e.g., \cite[2.3a]{ev}. 
Since $\Omega_P(k-2,0) \cong \O_P(k-4,0) \+ \O_P(k-2,-2)$, we  get  
\[
h^2(\Omega_P(k-2,0))=0 \quad \mbox{and} \quad
h^1(\Omega_P(k-2,0))=\begin{cases} 2, &
 \,\, \mbox{if}\quad  k=2  \\
 k-1, & \,\, \mbox{if}\quad  k > 2.  \end{cases}
\]
We compute
$h^1(\O_R(k-2,0))=h^0(\O_R(4-k,2))=h^0(\O_P(4-k,2))=\max\{0,15-3k\}$. Hence,
from Lemma \ref{lemma:nuovo24.11} and \eqref{eq:normdef2}, we obtain
\begin{equation}
\label{eq:cohcv}
h^1\bigl(\T_{\widetilde{S}}(-\widetilde{H})\bigr)
= h^1\bigl(\Omega_P(\log R)(k-2,0)\bigr)
= \max\{0,15-3k\}+\cork (\partial), 
\end{equation}
with $\partial$ the coboundary map 
$H^0(\O_{R}(k-2,0)) \to H^1(\Omega_P(k-2,0))$
of \eqref{eq:normdef2}.

When $k=2$, whence $g=5$, we have
\[ 
\partial: \CC \cong H^0(\O_R) \longrightarrow 
H^1( \Omega_P) \cong \CC^2,
\]
 which  is injective, its image being the $1$-dimensional  subspace
 of $H^{1,1}(P)$ generated by the class of $R$.
Thus, $\cork (\partial)=1$ and the lemma follows from
\eqref{eq:cohcv} and \eqref{eq:h1soprasotto}.

Similarly, by \eqref{eq:cohcv} and \eqref{eq:h1soprasotto} the lemma
follows when $k >2$ if we prove the surjectivity of
$\partial$.
It suffices to prove that its restriction 
to the image of the multiplication map
\[
H^0 (\O _P (k-2,0)) \otimes H^0(\O_R)
\longrightarrow H^0(\O_R (k-2,0))
\]
is surjective.
This restriction is the composed map
\begin{equation}
\label{restr:Halphadual}
H^0 (\O _P (k-2,0)) \otimes H^0(\O_R)
\xrightarrow{\phi_1} H^0 (\O _P (k-2,0)) \otimes H^1( \Omega_P)
\xrightarrow{\phi_2} H^1( \Omega_P (k-2,0)),
\end{equation}
where $\phi_1$ is the tensor product of the identity with the same map
$H^0(\O_R) \to H^1( \Omega_P) \cong H^{1,1}(P)$ as above, and
$\phi_2$ is defined by cup-product.

As we saw, the map  $\phi_1$ is injective, and its image is 
$H^0 (\O _P (k-2,0))\otimes \mathbb C\cdot[R]$, where $[R]$ is the class of $R$ in 
$H^ {1,1}(P)\cong H^1( \Omega_P)$. 
 By  the K\"unneth formula we have
\[
H^1( \Omega_P (k-2,0))\cong \pr_1^ * \big (H^ 0(\O_{\mathbb P^1}(k-2))  \big)\otimes \pr_2^* \big( H^ 1(\Omega_{\mathbb P^ 1})\big)
\] 
where $\pr_i: P\to \PP^ 1$, $1\leq i\leq 2$, are the two projections. Moreover
$
H^0 (\O _P (k-2,0))\cong \pr_1^ * \big (H^ 0(\O_{\mathbb P^1}(k-2))  \big).
$
Hence the map 
\[
\phi_2: \pr_1^ * \big (H^ 0(\O_{\mathbb P^1}(k-2)) \big)\otimes
H^1( \Omega_P)\longrightarrow \pr_1^ * \big (H^ 0(\O_{\mathbb
P^1}(k-2)) \big)\otimes \pr_2^* \big( H^ 1(\Omega_{\mathbb P^ 1})\big)
\]
is the tensor product of the identity on the first factor and of the
natural map
$
H^1( \Omega_P)\to \pr_2^ *\big( H^ 1(\Omega_{\mathbb P^ 1})\big),
$
which maps $\mathbb C \cdot [R]$ isomorphically to the target $\pr_2^ *\big( H^ 1(\Omega_{\mathbb P^ 1})\big)\cong \mathbb C$. Hence  $\phi_2$ maps the image of $\phi_1$  isomorphically  onto $H^1( \Omega_P (k-2,0))$, showing that the composed map
\eqref{restr:Halphadual} is surjective.  Thus, 
$\partial$  is surjective, which ends the proof.
\end{proof}

\begin{proof}[Proof of Proposition \ref{prop:fi2II}]
 (i)--(ii)  By Corollary \ref{cor:geninj} and Lemma \ref{lemma:dadimlemma1}, the sum of the dimensions of a general fiber of $\chi_{5,2}^{(II)^+}$ and a general fiber of $\chi_{5,2}^{(II)^-}$
is $10$.  Hence, assertions (i) and (ii)  follow by Lemma \ref{lemma:posfibfi2}(iii),(iv).   

(iii) This is a consequence of Corollary \ref{cor:geninj} and Lemma \ref{lemma:dadimlemma1}, as both $(S,H)$ and $(S,H+K_S)$ are general elements of $\E_{7,2}^{(II)}$.

(iv)--(v)  By Lemma \ref{lemma:diff} and Lemma \ref{lemma:mod} there are pairs $(S,H=4E_1+2E_2) \in \E_{9,2}^{(II)^+}$ such that $E_1$ and $E_2$ are nef and $h^1(\T_S(-H)) \geq 2$. Similarly, Lemma \ref{lemma:diff} and Corollary \ref{cor:EF1} imply that $h^1(\T_S(-(H+K_S))) \geq 1$. Hence, equality is attained in both cases by Lemma \ref{lemma:dadimlemma1}, whence also for general $(S,H) \in \E_{9,2}^{(II)^+}$. Corollary \ref{cor:geninj} yields the result.

(vi) This is an immediate consequence of Corollary \ref{cor:geninj} and Lemma \ref{lemma:dadimlemma1}.
 \end{proof}

 We next   prove  Theorem \ref{mainthm3}.  We recall that the moduli spaces $\E_{g,1}$ are all irreducible (cf.\  \cite{cdgk}). By Corollary \ref{cor:geninj}, the theorem is a consequence  of the following lemma.

\begin{lemma} \label{lemma:fi1-I+}  
 For  general $(S,H) \in \E_{g,1}$,  $g \geq 2$,  we have 
$ h^1(\T_S(-H))=\max \{0, 10-g\}$. 
\end{lemma}

\begin{proof} By \eqref {eq:h1soprasotto} and the fact that $\E_{g,1}$ is irreducible, it suffices to prove that $h^ 1(T_{\widetilde S}(-\widetilde H))=\max \{0,20-2g\}$.  
 
Consider $\sigma_1 \in H^0(\O_R(4,2))$ and $\sigma_2 \in H^0(\O_R(2,4))$  two sections (uniquely defined up to constants) whose zero schemes $Z(\sigma_1)=Z_1$ and $Z(\sigma_2)=Z_2$ are the ramification divisors of the $4:1$ maps $R \to \PP^1$ defined by the two projections of $P$ to $\PP^1$.  Note that $Z_1 \cap Z_2 = \emptyset$. Indeed a point in $Z_1 \cap Z_2$ would be singular for $R$, a contradiction.

We remark for later use that the scheme $Z_1 \in |\O_R(4,2)|=|\omega_R(2,0)|$ has length $24$ and consists of the ramification points of the first projection $R \to \PP^1$, thus of the points where the fibers in $|\O_{P}(1,0)|$ are tangent to $R$. On $\widetilde{S}$ these fibers become singular members of $|\widetilde{E}_1|$, that are mapped pairwise onto singular members of $|2E_1|$ on $S$. Thus, if $S$ is general, $Z_1$ consists of $24$ points on distinct elements of $|\O_{P}(1,0)|$, as  it is well-known that an elliptic pencil on a general Enriques surface has precisely $12$ singular reduced fibres, all nodal, cf., e.g., \cite[Thm. 4.8 and Rem. 4.9.1]{FM}. 
 
For  any   integer $ k \geq 1$, consider $H_{ k } \sim  k  E_1+E_2$.  Note that $H=H_{g-1}$. 
 
\begin{claim}\label{cl:1}
For every ${ k }\geq 1$, one has
\begin{equation}\label{eq:cacca}
 h^1(\T_{\widetilde{S}}(-\widetilde{H}_{ k })) =18-2{ k }+h^ 0(\O_P({ k }+2,1)\otimes \I_{Z_1}).
\end{equation}
\end{claim}

\begin{proof}[Proof of the Claim]
 We have an exact sequence (cf., e.g. \cite[2.3c]{ev})  
\[
  \xymatrix{0 \ar[r] & \Omega_P(\log R)({ k }-2,-1) \ar[r] &  \Omega_P({ k }+2,3) \ar[r]^{\gamma} &  \omega_R({ k }+2,3) \ar[r] & 0. }
\]
Since $\Omega_P({ k }+2,3) \cong \O_P({ k },3) \+ \O_P({ k }+2,1)$, we have
$h^1(\Omega_P({ k }+2,3))=0$, whence  
\begin{equation}
  \label{eq:cohcv002}
  h^1(\T_{\widetilde{S}}(-\widetilde{H}_{ k })) 
=h^1(\Omega_P(\log R)(k -2,-1))
=\cork H^0(\gamma)
\end{equation}
(where the left equality follows from Lemma
\ref{lemma:nuovo24.11}). Using the fact that $\omega_R \cong
\O_R(2,2)$, we may write $H^0(\gamma)$ as
\begin{equation*}
  \label{eq:cohcv03} H^0(\gamma): H^0(\O_P({ k },3)) \+ H^0(\O_P({ k }+2,1)) \longrightarrow H^0( \O_R({ k }+4,5)).
\end{equation*}
Moreover, computing dimensions yields that the  domain  has dimension $6{ k }+10$ and the target has dimension $4{ k }+28$, whence 
\begin{equation}
  \label{eq:kercoker}
  \cork (H^0(\gamma)) = 18-2{ k }+\dim (\ker H^0(\gamma)).
\end{equation}
We have
$H^0(\gamma)= \gamma_1+\gamma_2$, 
where
\begin{alignat*}{2}
H^0(\O_P({ k },3)) &\stackrel{\gamma_1}{\longrightarrow}
H^0( \O_R({ k }+4,5)), & \quad
 s & \stackrel{\gamma_1}{\longmapsto}
 s|_R \cdot \sigma_1 \\
H^0(\O_P({ k }+2,1)) &\stackrel{\gamma_2}{\longrightarrow} 
H^0( \O_R({ k }+4,5)), & 
t & \stackrel{\gamma_2}{\longmapsto}
t|_R \cdot \sigma_2.
\end{alignat*}
 The restriction  
$H^0(\O_P({ k },3)) \to H^0(\O_R({ k },3))$ is an isomorphism, as $h^0(\O_P({ k }-4,-1))=h^1(\O_P({ k }-4,-1))=0$.  Hence,  
$\gamma_1$ is injective and  $\im \gamma_1 = H^0(\O_R({ k }+4,5) \* \I_{Z_1})$. 
Since $h^0(\O_P(k-2,-3))=0$, the restriction map $H^0(\O_P({ k }+2,1)) \to H^0(\O_R({ k }+2,1))$ is injective (but not surjective). It follows that $\gamma_2$ is injective and 
\begin{equation}
\label{eq:interW}
\ker H^0(\gamma) \cong 
\left(\im \gamma_1 \cap \im \gamma_2\right) =
H^0(\O_P({ k }+2,1) \* \I_{Z_1}).
\end{equation}
Thus, \eqref {eq:cacca} follows from \eqref{eq:cohcv002},  \eqref{eq:kercoker} and \eqref {eq:interW}. \end{proof} 

\begin{claim}\label{claim:vacca} 
 If $1\leq  k \leq g-1$ and $(S,H)$ is general, then 
$h^ 0(\O_P(k+2,1)\otimes \I_{Z_1})$ is even. 
\end{claim}

\begin{proof} [Proof of the Claim]
By 
 \eqref {eq:h1soprasotto} written for $H_{ k }$,  the fact that $(S,H)$ is general (whence also all $(S,H_{ k })$ are general) and the fact that $\E_{{ k }+1,1}$ is irreducible, we have $h^1(T_S(-H_{ k }))=h^1(T_S(-H_{ k }+K_S))$, so that  $h^1(\T_{\widetilde{S}}(-\widetilde{H_{ k }}))$ is even. Hence the claim follows from \eqref {eq:cacca}. \end{proof}
 
\begin{claim}\label{claim:zacca} One has
$h^ 0(\O_P({ k }+2,1)\otimes \I_{Z_1}) =0$ for $1\leq { k } \leq 9$ and $(S,H)$ general.
 \end{claim}
 
 \begin{proof}[Proof of the Claim] 
Assume $h^ 0(\O_P({ k }+2,1)\otimes \I_{Z_1}) >0$. Then  $h^ 0(\O_P({ k }+2,1)\otimes \I_{Z_1})\geq 2$ by Claim \ref {claim:vacca}. Write $|\O_P({ k }+2,1)\otimes \I_{Z_1}|=M+\Delta$, where $M$ is the moving part and $\Delta$ the fixed part. 
 
Assume first that $\Delta$ contains an irreducible curve $B\in |\O_P(\beta,1)|$, for some $\beta\leq { k }+2$. Then $\Delta=B+F_1+\cdots+F_\alpha$, where $F_i\in |\O_P(1,0)|$ and $0\leq \alpha\leq { k }+2-\beta$. Hence $M$ consists of divisors in $|\O_P({ k }+2-\alpha-\beta,0)|$. Since $M$ has no fixed part, then $Z_1\subset \Delta$. Therefore $M=|\O_P({ k }+2-\alpha-\beta,0)|$ and $\Delta$ is the unique curve in $|\O_P(\alpha+\beta,1)\otimes \I_{Z_1}|$. In particular $h^ 0(\O_P(\alpha+\beta,1)\otimes \I_{Z_1})=1$. So Claim \ref {claim:vacca} implies that $\alpha+\beta\leq 2$. As $Z_1\subset R\in |\O_P(4,4)|$, then we must have $24=\deg(Z_1)\leq \O_P(\alpha+\beta,1)\cdot \O_P(4,4)=4(\alpha+\beta+1)\leq 12$, a contradiction. 
 
The remaining case is $\Delta=F_1+\cdots+F_\alpha$ where $F_i\in |\O_P(1,0)|$ and $0\leq \alpha\leq { k }+2$. Let $Z''$ be the largest subset of $Z_1$ contained in $\Delta$ and set $Z'=Z_1-Z''$. We thus have $M=|\O_P({ k }+2-\alpha,1)\otimes \I_{Z'}|$ and $\dim(M)=h^ 0(\O_P({ k }+2,1)\otimes \I_{Z_1})-1\geq 1$ by Claim \ref {claim:vacca}. As $M$ is base component free, it contains irreducible members. Hence  
$\deg (Z')\leq \O_P({ k }+2-\alpha,1)^ 2=2({ k }+2-\alpha)$. Since $\deg(Z'')\leq \alpha$, because the points of $Z_1$ lie in different elements of $|\O_P(1,0)|$, we have $2({ k }+2)\geq 2\alpha+\deg(Z')\geq 2\deg(Z'')+\deg (Z')\geq \deg(Z_1)=24$. Hence ${ k }\geq 10$, which proves  the claim. \end{proof}

We can now finish the proof of the lemma.  By  \eqref {eq:cacca} written for ${ k }=g-1$, we have 
\begin{equation}\label{eq:lof}
h^ 1(T_{\widetilde S}(-\widetilde H))=20-2g+h^ 0(\O_P(g+1,1)\otimes \I_{Z_1}).
\end{equation}

 Assume $g\leq 10$.   
By  Claim \ref {claim:zacca} we have $h^ 0(\O_P(g+1,1)\otimes \I_{Z_1})=0$,  so  $h^ 1(T_{\widetilde S}(-\widetilde H))  =20-2g$  by \eqref{eq:lof},  as  wanted. 

 Assume  $g\geq 11$.  For any $n\geq 0$  and  $F\in |\O_P(1,0)|$ such that $F\cap Z_1=\emptyset$, we have
\[
0\longrightarrow \O_P(n,1)\otimes \I_{Z_1}\longrightarrow \O_P(n+1,1)\otimes \I_{Z_1}\longrightarrow
\O_P(n+1,1)\otimes \O_F \cong \O_{\PP^ 1}(1)\longrightarrow 0,
\]
whence 
$h^ 0(\O_P(n+1,1)\otimes \I_{Z_1})\leq h^ 0(\O_P(n,1)\otimes \I_{Z_1})+2$.  
Arguing inductively, we have
\[
h^ 0(\O_P(g+1,1)\otimes \I_{Z_1})\leq h^ 0(\O_P(g-i,1)\otimes \I_{Z_1})+2(i+1)
\]
for every $i\in \{0,\ldots,g\}$. Setting $i=g-11$ and applying Claim \ref {claim:zacca} we get 
\[ h^ 0(\O_P(g+1,1)\otimes \I_{Z_1})  \leq  h^ 0(\O_P(11,1)\otimes \I_{Z_1})+2(g-11+1) =  2g-20. \] 
 Inserting in \eqref {eq:lof} we get 
$h^ 1(T_{\widetilde S}(-\widetilde H))= 0$, as wanted.  \end{proof}

%
%

\end{document}